\documentclass[10pt]{amsart}

\usepackage{amsthm}
\usepackage{amsmath}
\usepackage{amssymb, dsfont}
\usepackage[mathscr]{euscript}
\usepackage{mathrsfs}
\usepackage{hyperref}

\usepackage{tikz-cd}

\usepackage[all]{xy}

\usepackage{latexsym}
\usepackage{framed}

\usepackage{mathabx}

\usepackage{bm}

\title[Flag Space, Matroidal Schur Algebras and the Steinberg Representation]{Flag Space, Matroidal Schur Algebras and the Steinberg Representation}

\author{Carl Mautner}

\address{Carl Mautner \\
Department of Mathematics, University of California, Riverside}
\email{mautner@math.ucr.edu}

\DeclareFontFamily{U}{mathx}{\hyphenchar\font45}
\DeclareFontShape{U}{mathx}{m}{n}{
      <5> <6> <7> <8> <9> <10>
      <10.95> <12> <14.4> <17.28> <20.74> <24.88>
      mathx10
      }{}
\DeclareSymbolFont{mathx}{U}{mathx}{m}{n}
\DeclareFontSubstitution{U}{mathx}{m}{n}
\DeclareMathAccent{\widecheck}{\mathord}{mathx}{"71} 
 
\newtheorem{theorem}{Theorem}[section]
\newtheorem{lemma}[theorem]{Lemma}
\newtheorem{proposition}[theorem]{Proposition}
\newtheorem{corollary}[theorem]{Corollary}

\newtheorem{conjecture}[theorem]{Conjecture}

{%\theorembodyfont{\normalfont}
\theoremstyle{definition}
\newtheorem{definition}[theorem]{Definition}
\newtheorem{example}[theorem]{Example}
\newtheorem*{example*}{Example}

\newtheorem{remark}[theorem]{Remark}
}

\newcommand{\excise}[1]{}

\renewcommand{\Im}{\operatorname{Im}}
\newcommand{\im}{\operatorname{im}}

\newcommand{\Ker}{\operatorname{Ker}}
\renewcommand{\dim}{\operatorname{dim}}
\newcommand{\codim}{\operatorname{codim}}

\newcommand{\Hom}{\operatorname{Hom}}
\newcommand{\End}{\operatorname{End}}

\newcommand{\C}{{\mathbb{C}}}
\newcommand{\Z}{{\mathbb{Z}}}

\newcommand{\FM}{{\mathbb{F}}}

\newcommand{\M}{\mathfrak{M}}

\newcommand{\B}{\mathcal{B}}
\newcommand{\Bt}{\tilde\B}

\newcommand{\F}{\mathcal{F}}

\newcommand{\Lc}{\mathcal{L}}

\newcommand{\onto}{\twoheadrightarrow}

\newcommand{\cal}{\mathcal}

\newcommand{\cB}{{\cal B}}

\newcommand{\cI}{{\cal I}}
\newcommand{\cL}{{\cal L}}

\newcommand{\cT}{{\cal{T}}}

\newcommand{\md}{\mathrm{-mod}}

\renewcommand{\setminus}{\smallsetminus}

\newcommand{\sgn}{\operatorname{sgn}}

\newcommand{\ul}{\underline}

\DeclareMathOperator{\St}{St}

\newcommand{\Bas}{\mathop{\mathrm{Bas}}\nolimits}

\newcommand{\cUc}{{\widecheck{\mathcal U}}}
\newcommand{\Uc}{\widecheck{U}}
\newcommand{\Rc}{\widecheck{R}}

\newcommand{\uc}{{\check{u}}}

\newcommand{\Fl}{\mathscr{F}\ell}
\newcommand{\FlP}{\mathscr{F}\ell\mathscr{P}}

\newcommand{\Fs}{\mathscr{F}}

\newcommand{\Bell}{\ensuremath{\boldsymbol\ell}}
\newcommand{\e}{\mathfrak e}

\begin{document}

\maketitle
\begin{abstract}
We define via generators and relations an extended version of the matroidal Schur algebras introduced in earlier work of the author with Tom Braden.  In the case of projective geometries, we relate the representation theory of the resulting algebras to the structure of the modular Steinberg representation of $GL(n,q)$.  In the general case we show that the extended matroidal Schur algebras have a triangular, or Reedy, decomposition and are thus quasi-hereditary.  We also obtain a new description of the original matroidal Schur algebras and clarify the connection to the flag spaces of Brylawski--Varchenko.
\end{abstract}

%%%%%%%%%%%%%%%%%%%
\section{Introduction}
%%%%%%%%%%%%%%%%%%%

Let $k$ be a field of characteristic $\ell$.  The Schur algebra $S_k(n,d)$ is a quasi-hereditary algebra that connects the representation theory over $k$ of the symmetric and general linear groups.  In~\cite{Mautner} the author observed that the category of $S_k(n,n)$-modules is equivalent to a category of perverse sheaves on the nilpotent cone $\mathcal N \subset \mathfrak{gl}_n(\C)$.
In~\cite{BMHyperRingel} Braden and the author studied perverse sheaves on affine hypertoric varieties as a hypertoric analogue of the category of $S_k(n,n)$-modules and discovered~\cite{BMmatroid} a family of Schur-like quasi-hereditary algebras associated to matroids.

A \textit{matroid} $M = (E,\cL)$ is a finite set $E$ together with a nonempty collection $\cL = \cL(M)$ of subsets of $E$, called the \textit{flats} of $M$, that satisfy the following two axioms:
\begin{enumerate}
\item The intersection of any two flats is again a flat, and
\item If $F \in \cL$ and $x \in E \setminus F$, then there exists a unique minimal flat $G$ such that $F \cup \{x\} \subseteq G$.
\end{enumerate}

For any matroid $M$, Braden and the author~\cite{BMmatroid} defined a pair of finite dimensional $k$-algebras $R_k(M)$ and $\Rc_k(M)$ that are quasihereditary and Ringel dual to each other.  The algebra $R_k(M)$ is naturally isomorphic to $\Rc_k(M^*)$, where $M^*$ denotes the dual matroid.  Like the Schur algebra, the algebras $R_k(M)$ and $\Rc_k(M)$ are semisimple if and only if $\ell$ lies outside of an explicit finite set of prime numbers that depends on the combinatorics of the matroid.   We named these algebras \textit{matroidal Schur algebras}.

There were, however, two unsatisfying features of our work:
\begin{enumerate}
\item The definitions of $R_k(M)$ and $\Rc_k(M)$ were unwieldy:  we first defined a large vector space $\B$ (bigraded by the poset of cyclic flats of $M$) equipped with an orthonormal basis and two families of operators and their adjoints.  The matroidal Schur algebras were then defined as the subalgebras of $\End(\B)$ generated by one or the other families of operators and their adjoints.  In particular, we did not give a generators and relations presentation.
\item Matroidal Schur algebras, a priori, lacked a clear connection to known representation theory: matroidal Schur algebras were discovered by following a geometric analogy suggested by the theory of symplectic singularities and while the representation theory of matroidal Schur algebras and the classical Schur algebra share some common features (e.g., quasihereditary, semisimple for all but finitely many characteristics), direct links to the representation theory of $S_k(n,n)$ or the general linear group were by analogy only.
\end{enumerate}

In the current paper we address both of these points.  First, we define a new, extended matroidal Schur algebra $S_k(M)$ with a simple presentation by generators and relations and use it to give a cleaner description of the original algebras $R_k(M)$ and $\Rc_k(M)$.  Second, we begin the study of two classes of examples which give the first direct connections to known representation theory.  For example, we relate the matroidal Schur algebras associated to projective geometries over $\FM_q$ to the structure of Steinberg representations of $k[GL_n(\FM_q)]$.  As the collection of matroids is vast, we imagine there could be many more interesting connections.

\medskip

We begin in Section~\ref{sec-defn} by introducing our central object of study, the extended matroidal Schur algebra $S_k(M)$ (and its integral form $S(M)$).  Later in the paper we show that this new matroidal Schur algebra $S_k(M)$ is quasihereditary (in fact, triangular or \textit{Reedy}) with highest weight poset equal to $\Lc^{\mathrm{op}}$ the lattice of all flats, but ordered by inverse inclusion.  Thus for every flat $F \in \Lc$, there is a standard module $\Delta_{F,k}$ with simple quotient $L_{F,k}$.

Before studying the case of general matroids, however, we highlight two special cases accessible to those unfamiliar with matroid theory.  In Section~\ref{sec-projgeom} we consider the projective geometry matroidal Schur algebra $S_k(PG(n-1,q))$.  In this case, flats are labelled by the linear subspaces of $\FM_q^n$ and so the standard and simple $S_k(PG(n-1,q))$-modules are parametrized up to isomorphism by the set of linear subspaces $V \subseteq \FM_q^n$.  The algebra  $S_k(PG(n-1,q))$ is semisimple if and only if $\ell$ does not divide $\prod_{i=1}^{n} \frac{q^i-1}{q-1}=|GL_n(\FM_q)/B|$, the number of $\FM_q$-points of the flag variety.  Each $S_k(PG(n-1,q))$-module $X$ has a weight space decomposition $X = \bigoplus_{V \subseteq \FM_q^n} X^V$.  

In Section~\ref{sec-Steinberg} we prove: 

\begin{theorem}
Suppose that $\ell \nmid q$ and $V \subseteq W$ are subspaces of $\FM_q^n$.  Then:
\begin{enumerate}
\item The $W$-weight space $\Delta_{V,k}^W$ of the standard module $\Delta_{V,k}$ is naturally a $k[GL(W/V)]$-module isomorphic to the dual of the modular Steinberg representation $\St^k(W/V)^*$.
\item the $W$-weight space $L_{V,k}^W$ of the simple module $L_{V,k}$ is isomorphic (as $k[GL(W/V)]$-modules) to the socle of the modular Steinberg representation $\St^k(W/V)$ of $GL(W/V)$ (which is self-dual).
\end{enumerate}
\end{theorem}

We note that Geck~\cite[3.6]{Geck} identified the partition of $n$ corresponding the socle of the Steinberg representation in James’  parametrization of the unipotent simple modules of $G$ (see Theorem~\ref{thm-geck}).  We direct the reader to Theorem~\ref{thm-gowfilt} and its corollaries for the equal characteristic case and more comprehensive results about the structure of the standard $S_k(PG(n-1,q))$-modules.  

In Section~\ref{sec-Kn} we present a second special case, the complete graph Schur algebra $S_k(M(K_n))$.  In this case, the lattice of flats $\Lc(M(K_n))$ can be defined with the lattice of set partitions of $\{1,\ldots,n\}$ and its representation theory is related to the structure of the $\mathfrak S_n$-representation $\mathrm{Lie}_k(n)$ on the multilinear component of the free Lie algebra over $k$ on $n$ generators.

In Section~\ref{sec-matroids} we prepare to return to the general case by recalling matroid terminology.  Matroid experts may wish to skip this section.

Section~\ref{sec-reedy} provides a triangular or \textit{Reedy} decomposition of $S_k(M)$.

In Section~\ref{sec-mult} we define an explicit cellular basis for $S_k(M)$ and give a formula for the multiplication of basis elements.

In Section~\ref{sec-modules} we introduce the standard $S(M)$-module $\Delta_F$ for each $F \in \Lc(M)$, which is isomorphic (as a free $\Z$-module) to the flag space of Brylawski--Varchenko~\cite{BV,SV} for the contraction $M_F$ (and equivalently isomorphic to the linear dual to the Orlik-Solomon algebra of $M_F$).  In particular, we show:

\begin{theorem}\label{thm-results} For each flat $F$ of $M$ the standard module $\Delta_{F,k} = \Delta_F \otimes_\Z k$ is naturally non-negatively graded and the graded dimension of $\Delta_{F,k}$ is given by:
\[ \sum_{G \supseteq F} \mu^+(F,G) t^{\rho(G)-\rho(F)} = (-t)^{\rho(M_F)} p(M_F;-t^{-1})\]
where the sum runs over all flats $G$ that contain $F$, $p(M;t)$ denotes the characteristic polynomial of $M$ with variable $t$.

The graded dimension of $S_k(M)$ is therefore:
\[ \sum_{F\in \Lc(M)} p(M_F;-t) p(M_F;-t^{-1}),\]
where the sum runs over all flats $F$ of $M$.
\end{theorem}

If $\ell = \mathrm{char}~k>0$, using the bilinear form $\langle~,~\rangle_F: \Delta_F \times \Delta_F \to \Z$ induced by the cell structure, the standard module $\Delta_{F,k}= \Delta_F \otimes k$ has a natural `Jantzen' filtration and a simple quotient $L_{F,k} = \Delta_{F,k}/\mathrm{rad}~\langle~,~\rangle_{F,k}$, the set of which form a complete set of non-isomorphic simple $S_k(M)$-modules.

Under the identification of the standard module $\Delta_F$ with flag space, the bilinear form $\langle~,~\rangle_F$ coincides with Brylawski--Varchenko's bilinear form on flag space.  As a corollary of the determinant formula of~\cite{BV} we have:

\begin{theorem}
Let $\ell$ be the characteristic of $k$.  The standard module $\Delta_{F,k}$ is simple if and only if $\ell$ does not divide $|K \setminus F|$ for any flat $K \supset F$ such that $M^K_F$ is connected.\footnote{We write $M^K$ for the restriction of $M$ to $K$ and $M^K_F = (M^K)_F$.}

In particular, $S_k(M)$ is semisimple if and only if $\ell \nmid |K \setminus F|$ for any pair of flats $K \supset F$ such that $M^K_F$ is connected.
\end{theorem}

In Section~\ref{sec-tilting} we construct a natural action of $S(M)$ on $IN(M)$, the independence complex of $M$, and more generally $IN(M_F)$ for every $F \in \Lc(M)$.  These actions play a key role in relating the algebras $S(M)$ and $\Rc(M)$.  I expect to show in future work that these modules are tilting.

\begin{remark}  The original matroidal Schur algebras have the remarkable property that the Ringel dual of $\Rc(M)$ is $\Rc(M^*)$, which is evidenced by the fact that their highest weight posets are opposite: the lattice of cyclic flats of $M$ is the opposite of the lattice of cyclic flats of $M^*$.  For the extended matroidal Schur algebra $S(M)$ the highest weight poset is the lattice of all flats of $M$, which is not the opposite of the lattice of flats of $M^*$ in general.  Thus the Ringel dual of $S(M)$ is usually not $S(M^*)$.
\end{remark}

The definition of $\Rc(M)$ is recalled in Section~\ref{sec-schur}. Theorem~\ref{thm-isom} asserts that there is an analogous description of $S(M)$, which we prove in Section~\ref{sec-faithful}.  Theorem~\ref{thm-quiver}, which is proved in Section~\ref{sec-Rcheck}, describes $\Rc(M)$ as the quotient of $eS(M)e$ by an additional relation, where $e$ is an idempotent.

\begin{remark}
Geometrically, when $M$ is the matroid underlying a unimodular hyperplane arrangement $H$ with associated affine hypertoric variety $\M = \M(H)$, Braden and the author proved~\cite{BMHyperRingel, BMmatroid} that the category of modules over $\Rc_k(M)$ is equivalent to the category of torus-equivariant perverse sheaves with $k$-coefficients on $\M$, constructible with respect to the stratification by \textit{cyclic} flats, and $R_k(M)$ is isomorphic to a convolution algebra on the Borel-Moore cohomology associated to resolutions of sub-hypertoric varieties of $\M$.
\end{remark}

It is natural to ask about the geometric significance of the extended matroidal Schur algebra $S_k(M)$ in the unimodular hypertoric setting.  I expect:

\begin{conjecture}
Let $H$ be a unimodular hyperplane arrangement with corresponding affine hypertoric variety $\M$ and underlying matroid $M$.  Then the category of torus-equivariant perverse sheaves with $k$-coefficients on $\M$, constructible with respect to the stratification by all flats, is equivalent $S_k(M)\md$.
\end{conjecture}

\subsection{Acknowledgements}
I am grateful to George Santellano, Wee Liang Gan, Sasha Kleshchev and Steffen Koenig for their interest and questions.

%%%%%%%%%%
\section{Extended matroidal Schur algebra}
\label{sec-defn}
%%%%%%%%%%

For a matroid $M$, the set of flats $\cL = \cL(M)$ ordered by inclusion is a geometric lattice with a rank function $\rho$.  For $K,L \in \cL$, we say that $L$ covers $K$ and write $L \gtrdot K$ (or $K \lessdot L$) if $L$ is minimal among flats containing $K$.  In particular, this means that $\rho(L) = \rho(K) +1$ whenever $L\gtrdot K$.

We call a (finite) sequence of flats $(K_1, \ldots, K_l)$ in $M$ a \textit{flag path} if for each $i = 1, \ldots, l-1$ either $K_i \lessdot K_{i+1}$ or $K_i \gtrdot K_{i+1}$.\footnote{Note that one could alternatively describe a flag path in $M$ as a path in the quiver with vertices labelled by the flats of $M$ and an edge from $X$ to $Y$ if $X \lessdot Y$ or $X \gtrdot Y$.}  We will often replace the commas in the sequence with the appropriate symbols $\gtrdot$ or $\lessdot$ to make clear the relative sizes of the flats.

\begin{definition} Let $\FlP(M)$ denote the free $\Z$-module on the set of all flag paths, viewed as an algebra with the concatenation product:
\[ (K_1, \ldots, K_l) \cdot (L_1, \ldots, L_{l'}) = \begin{cases}
(K_1, \ldots, K_l=L_1,L_2,\ldots,L_{l'}) & \mathrm{if}~ K_l=L_1, \\
0 & \mathrm{otherwise.}
\end{cases}
\]
The \textit{(extended) matroidal Schur algebra} $S(M)$ is the quotient of $\FlP(M)$ by the following two types of relations:
\begin{enumerate}
\item (\textit{Flag space relations}) Suppose $L < K$ and $\rho(K) = \rho(L)+2$.  Then
\[ \sum_{X} (L \lessdot X \lessdot K) = 0 = \sum_{X} (K \gtrdot X \gtrdot L),\label{eq-FS} \tag{FS}\]
where the sums are over all flats $X$ such that $K \gtrdot X \gtrdot L$.

\item (\textit{Up-down relations}) Suppose $K \lessdot L$ and $K' \lessdot L$.  Then
\[ (K \lessdot L \gtrdot K') = \begin{cases}
| L \setminus K| (K) & \mathrm{if}~ K=K', \\
- (K \gtrdot K \cap K' \lessdot K') & \mathrm{if}~K\neq K'~\mathrm{and}~K \cap K' \lessdot K, \\
0 & \mathrm{otherwise.}
\end{cases}
\label{eq-vee}\tag{$\bigwedge$}\]
\end{enumerate}
We let $S_k(M):= S(M) \otimes_\Z k$.
\end{definition}

Note that $S(M)$ has a natural $\Z$-grading where $K\gtrdot L$ has degree $1$ and $L\lessdot K$ has degree $-1$.

%%%%%%%%%%
\section{Projective geometry Schur algebras}
\label{sec-projgeom}
%%%%%%%%%%

In this section we restrict our attention to the matroidal Schur algebras $S(PG(n-1,q))$ associated to projective geometries and state our general structural results in this special case.  We begin with this example because the definition and basic results can be stated using only standard notions from linear algebra and, as we show in Section~\ref{sec-Steinberg}, its representation theory is closely related to the structure of the Steinberg module for $GL_n(\FM_q)$.

Let $q$ be a prime power and $n$ be a natural number.  We recall:
\begin{definition}
Let $E = \mathbb P^{n-1}(\FM_q)$, the set of 1-dimensional subspaces of $\FM_q^n$. The \textit{projective geometry} $PG(n-1,q)$ is the matroid with underlying set $E$ and $\Lc$ is the set of subspaces $V \subseteq \FM_q^n$, where we view a subspace as corresponding to the set of all 1-dimensional subspaces it contains.
\end{definition}

Here the rank of a flat is the dimension of the corresponding subspace.  For subspaces $V$ and $W$ of $\FM_q^n$, $V$ is covered by $W$ if $V \subset W$ and $\dim W = \dim V +1$.

Thus a \textit{flag path} for $PG(n-1,q)$ is a sequence of subspaces $(V_1, \ldots,V_l)$ such that for each $i = 1, \ldots, l-1$, $V_{i}  \lessdot V_{i+1}$ or $V_i \gtrdot V_{i+1}$.

In this case our definition specializes to:

\begin{definition} The \textit{projective geometry Schur algebra} $S(PG(n-1,q))$ is the quotient of $\FlP(PG(n-1,q))$ by the following two types of relations:
\begin{enumerate}
\item (\textit{Flag space relations}) Suppose $V \subset W$ are subspaces of $\FM_q^n$ and $\dim(W) = \dim(V)+2$.  Then
\[ \sum_{X} (V \lessdot X \lessdot W) = 0 = \sum_{X} (W \gtrdot X \gtrdot V), \tag{FS}\]
where the sums are over all subspaces $X$ such that $V \lessdot X \lessdot W$.

\item (\textit{Up-down relations}) Suppose that $V, V'$ and $W$ are subspaces of $\FM_q^n$ such that  $V \lessdot W \gtrdot V'$ and $\dim V = k$.  Then
\[ (V \lessdot W \gtrdot V') = \begin{cases}
q^k(V) & \mathrm{if}~ V=V', \\
- (V \gtrdot V\cap V' \lessdot V) & \mathrm{if}~V \neq V'.
\end{cases}
\tag{$\bigwedge$}\]
\end{enumerate}
\end{definition}

For any subspace $V \subseteq \FM_q^n$, $(V)\in S(PG(n-1,q))$ is an idempotent.  For any $S(PG(n-1,q))$-module $X$ and we can consider the \textit{weight space} $X^V:= (V)\cdot X \subseteq X$ and the corresponding decomposition $X = \bigoplus_{V \subseteq \FM_q^n} X^V$.

\subsection{Triangular decomposition and graded rank} Our general results of Section~\ref{sec-reedy} imply in this case that $S_k(PG(n-1,q))$ is quasi-hereditary with highest weight poset given by the lattice of subspaces of $\FM_q^n$ ordered by inverse inclusion.  The proof involves showing the existence of the following triangular decomposition.

For each subspace $V \subseteq \FM_q^n$, let $\Fl^+_V$ (resp. $\Fl^-_V$) be the free $\Z$-module with basis the set of all flags $(V_l \gtrdot V_{l-1} \gtrdot \ldots \gtrdot V_1 \gtrdot V)$ (resp. $(V \lessdot V_{1} \lessdot \ldots \lessdot V_{l-1} \gtrdot V_l)$), where $l$ is allowed to vary.  Let $\Fs^+_V$ (resp. $\Fs^-_V$) be the quotient of $\Fl^+_V$ (resp. $\Fl^-_V$) by all flag space relations.
For each subspace $W \supseteq V$, we write $(\Fs^+)_V^W \subset \Fs^+_V$ (resp. $(\Fs^-)_V^W \subset \Fs^-_V$) for the submodule generated by flags $(W \gtrdot V_{l-1} \gtrdot \ldots \gtrdot V_1 \gtrdot V)$ (resp. the reverse).  In particular \[ \Fs^+_V = \bigoplus_{V\subseteq W} (\Fs^+)_V^W \quad \mathrm{and} \quad \Fs^-_V = \bigoplus_{V\subseteq W} (\Fs^-)_V^W.\]

Our results in Section~\ref{sec-reedy} imply that multiplication induces an isomorphism
\[ \bigoplus_{W' \supseteq V \subseteq W} (\Fs^+)_V^{W'} \otimes (\Fs^-)_V^W \cong \bigoplus_{V} \Fs^+_V \otimes \Fs^-_V \cong S(PG(n-1,q)).\]

One can then construct a (graded) standard module $\Delta_V$ for $S(PG(n-1,q))$ for each subspace $V \subseteq \FM_q^n$ such that 
$ \Delta_V \cong \Fs_V^+.$

Applying the results of Section~\ref{sec-modules} to this case, if $i=\dim V$, the graded rank of $\Delta_V$ is given by:
\[ \sum_{W \supseteq V} q^{\binom{\codim(V, W)}{2}} t^{\codim(V, W)} = \sum_{j=0}^{n-i}  q^{\binom{j}{2}} \binom{n-i}{j}_q t^{j} = \prod_{j=0}^{n-i-1} (1+q^j t)\]
where the first sum runs over all subspaces $W \subseteq \FM_q^n$ such that $V \subseteq W$ and in the second sum $\binom{n}{j}_q$ denotes the $q$-binomial coefficient, or equivalently the number of $j$ dimensional subspaces of $\FM_q^n$.

The graded rank of $S(PG(n-1,q))$ is therefore:
\[\sum_{i=0}^n \binom{n}{i}_q \left(\prod_{j=0}^{n-i-1} (1+q^j t)(1+q^j t^{-1}) \right) .\]

\subsection{Bilinear form on standard modules} In this example, the symmetric bilinear form defined in Section~\ref{sec-modules} on standard modules \[ \langle~,~\rangle_V : \Delta_V \times \Delta_V \to \Z,\]
is defined on flags $V_\bullet = (V_l \gtrdot \ldots \gtrdot V_1 \gtrdot V),~~W_\bullet = (W_m \gtrdot \ldots \gtrdot W_1 \gtrdot V)$
 by the equation
 \[ \langle V_\bullet , W_\bullet \rangle_V (V) = (V, W_1,\ldots, W_m) \cdot (V_l, V_{l-1}, \ldots, V_1,V),\]
so that  $\langle V_\bullet , W_\bullet \rangle_V$ is the scalar multiple of $(V)$ obtained from the product $(V, W_1,\ldots, W_m) \cdot (V_l, V_{l-1}, \ldots, V_1,V) \in S(PG(n-1,q))$.\footnote{To see that this product is indeed a scalar multiple of $(V)$, note that one can repeatedly apply the up-down relation and that at each step, all of the subspace involved contain $V$, so the expression will eventually reduce to a multiple of $V$.}  Note that the weight decomposition $\Fs^+_V = \bigoplus_{W\supseteq V} (\Fs^+)_V^W$ is orthogonal with respect to this form.  

Corollary~\ref{cor-flagmult} gives a more explicit description of this bilinear form. Namely for $V_\bullet, W_\bullet$ as above:
\begin{equation}\label{eq-form}
\langle V_\bullet , W_\bullet \rangle_V = \begin{cases}
\sgn(\sigma) q^{m \dim V + \Bell(w_0 \sigma)} & \mathrm{if}~W_m = V_l \\
0 & \textrm{otherwise,} \end{cases}
\end{equation}
where $\sigma = \sigma( V_\bullet , W_\bullet) \in \mathfrak{S}_m$ is the relative position of the flags (viewed as flags in $W_m/V$), $w_0$ is the longest element of $\mathfrak S_m$ (i.e., $w_0(1)=n, w_0(2) = n-1, \ldots, w_0(n)=1$) and $\Bell$ denotes the length function on $\mathfrak S_m$.

\subsection{Simple modules and Jantzen filtration} Let $\Delta_{V,k} = \Delta_V \otimes_\Z k$ be the standard $S_k(PG(n-1,q))$-module.  The integral bilinear form $\langle~,~\rangle_V : \Delta_V \times \Delta_V \to \Z$ induces a form 
\[\langle~,~\rangle_{V,k} : \Delta_{V,k} \times \Delta_{V,k} \to k.\]  By the theory of quasi-hereditary algebras, the standard module $\Delta_{V,k}= \Delta_V \otimes k$ has a simple quotient $L_{V,k}:= \Delta_{V,k}/\mathrm{rad}\langle~,~\rangle_{V,k}$ and these form a complete set of non-isomorphic simple $S_k(PG(n-1,q))$-modules.

Additionally, if $\ell = \mathrm{char}~k>0$, 
one can construct a natural `Jantzen' filtration 
\[\Delta_{V,k} = (\Delta_{V,k})^0 \supset (\Delta_{V,k})^1 \supseteq \ldots \supseteq (\Delta_{V,k})^N=0,\]
where $(\Delta_{V,k})^1 = \mathrm{rad}\langle~,~\rangle_{V,k}$. The filtration is defined by letting 
\[\Delta_V^i = \{ x \in \Delta_V~|~\langle x, v \rangle \in \ell^i \Z~\mathrm{for~all}~v\in \Delta_V \},\]
and $(\Delta_{V,k})^i$ be the image of $\Delta_V^i \otimes_\Z k$ in $\Delta_{V,k} = \Delta_V \otimes_\Z k$.

\subsection{Non-semimple characteristics} Recall that $k$ is a field of characteristic $\ell$.  As a consequence of Theorem~\ref{thm-det} (which is based on the determinant formula of Brylawski--Varchenko~\cite{BV,SV}), the standard module $\Delta_{V,k}$ associated to a $d$-dimensional subspace $V \subseteq \FM_q^n$  is simple if and only if $\ell$ does not divide $\frac{q^m-q^d}{q-1}$ for any $d<m\leq n$.   Equivalently, $\Delta_{V,k}$ is simple if and only if $\ell$ does not divide $q^{(n-d)d}\prod_{i=1}^{n-d} \frac{q^{i}-1}{q-1}$.

In particular, $\Delta_{\{0\},k}$ is simple if and only if $\ell$ does not divide $\prod_{i=1}^{n} \frac{q^i-1}{q-1}=|GL(n,q)/B|$, where $B$ is the subgroups of upper-triangular matrices.

It follows that $S_k(M)$ is semisimple if and only if $$\ell \nmid |\mathrm{GL}(n,q)|/(q-1)^n.$$

%%%%%%%%%%%%%%
\section{$S(PG(n-1,q))$ and Steinberg representations}\label{sec-Steinberg}
%%%%%%%%%%%%%%

In this section we recall the construction of the Steinberg representation $\St^k(\FM_q^n)$ of $GL_n(\FM_q)$ over $k$ and describe the $S_k(PG(n-1,q))$-modules $\Delta_{V,k}$ and $L_{V,k}$ in terms of the structure of the Steinberg modules $\St^k(\FM_q^m)$ and their socles for $m\leq n$.  As a consequence we give an explicit description of the characters of the simple $S_k(PG(n-1,q))$-modules.

\begin{definition}\cite{Steinberg}
The \textit{Steinberg module} $\St^k(\FM_q^n)$ for $G=GL_n(\FM_q)$ over $k$ is the left ideal $kG\e$ of the group algebra $kG$, where \[\e := \left( \sum_{\sigma \in \mathfrak S_n} \sgn(\sigma) \sigma \right) \sum_{b \in B} b,\]
and $B \subset G$ denotes the subgroup of upper triangular matrices and $\mathfrak S_n$ is viewed as a subgroup of $G$ in the usual way.
\end{definition}

Steinberg~\cite{Steinberg} showed that $\St^k(\FM_q^n)$ has a basis $\{u\e\}_{u \in U}$ where $u$ runs over upper-triangular unipotent matrices.

\subsection{Gow's filtration and the socle of $\St^k(\FM_q^n)$} Gow~\cite{Gow} considered the non-degenerate bilinear form on the group algebra of $G$ defined by $(g,h) \mapsto \delta_{g,h}$ and defined a symmetric bilinear $f$ on $\St_n^k(\FM_q)$ by restricting the form on the group algebra and dividing by $|B|$.  He showed that:
\begin{equation}\label{eq-f} 
f(u_1\e, u_2\e) = c_W (u^{-1}_1 u_2),
\end{equation}
where $c_W (u) = |\{ w \in \mathfrak S_n ~|~ wuw^{-1} \in U \}|.$

Gow then studied the resulting Jantzen filtration of $kG$-modules on $\St_n^k(\FM_q)$ when $k$ is a field of positive characteristic:
\[ \St_n^k(\FM_q) = I(0) \supset I(1) \supseteq I(2) \supseteq \ldots \]

Consider the $kG$-modules $M(i) = I(i)/I(i+1)$ for each $i\geq 0$.  Recall that $\ell$ denotes the characteristic of $k$.  Gow proved the following results:

\begin{theorem} 
\begin{enumerate}
\item \cite[4.3]{Gow} If $M(i)\neq 0$, then $M(i)$ is a self-dual $kG$-module.
\item \cite[6.2]{Gow} The subquotient $M(i)$ is non-zero if and only if $i$ is the $\ell$-valuation of $|G/P|$ for some parabolic subgroup $P$ of $G$, or equivalently the $\ell$-valuation of some $q$-multinomial coefficient $\binom{n}{\lambda_1,\ldots,\lambda_m}_q$ where $\lambda= (\lambda_1,\ldots,\lambda_m)$ is a partition of $n$.  
\item \cite[4.5]{Gow} If $\kappa$ is the $\ell$-valuation of $|G/B|$, then $M(\kappa)$ is the socle of $\St^k(\FM_q^n)$. 
\end{enumerate}
\end{theorem}

As the Steinberg module $\St^k(\FM_q^n)$ can be viewed as a submodule of $kG(\sum_{b \in B} b)$, the composition factors are examples of \textit{unipotent} representations of $G$, which James parametrized by partitions of $n$.  Building on results of Tinberg, Hiss, Khammash and Gow, Geck~\cite{Geck} identified the partition corresponding to the socle of the Steinberg module as follows.  Let $ \mathrm{char}~k = \ell >0$ and 
\[ e := \min \{ i \geq 2~|~1+q+q^2+\ldots+ q^{i-1} \equiv 0~ \mathrm{mod}~ \ell\}.\]

\begin{theorem}\label{thm-geck}\cite[3.6]{Geck}
The socle of $\St^k(\FM_q^n)$ is labelled by the partition $\mu_0 = \mu(n,e)$ of $n$ given by $m+1$ with multiplicity $r$ and $m$ with multiplicity $e-r-1$, where 
\[n = (e - 1)m + r,~\mathrm{for}~0 \leq r < e - 1.\]
\end{theorem}

\begin{remark}\label{rem-kham} For example, if $\ell$ divides $q+1$, then $e=2, m=n, r=0$ and the socle is the trivial module, a fact first proved by Khammash~\cite{Khammash} and re-proved by Gow~\cite[Theorem 4.7]{Gow}.
\end{remark}

\subsection{Homological description of the Steinberg module and dual bases} To relate Gow's Jantzen filtration with the Jantzen filtration of $\Delta_{V,k}$, we consider a different description of the Steinberg module.  For any $\FM_q$-vector space $V$, let $\cT(V)$ denote the Tits building of $V$, meaning the geometric realization of the poset of proper non-zero subspaces of $V$.  A theorem of Solomon states that $\cT(V)$ is homotopy equivalent to a wedge of $(\dim V-2)$-spheres and so the reduced homology of $\cT(V)$ with coefficients in any ring $k$ is concentrated in degree $\dim V-2$.

More concretely, $\tilde{H}_{\dim V-2}(\cT(V);k)$ can be obtained as the $k$-submodule of the vector space $\Fl(V)$ with basis the set of full flags $(V \gtrdot V_{\dim V-1} \gtrdot \ldots \gtrdot V_1 \gtrdot \{0\})$, given by the kernel of the map $\partial$ to the space of flags in $V$ with a gap defined by:
\[ \partial(V  \gtrdot V_{\dim V-1} \gtrdot \ldots  \gtrdot V_1 \gtrdot \{0\}) = \sum_{i=1}^{n-1} (-1)^i (V \gtrdot \ldots \gtrdot \widehat{V_{i}} \gtrdot \ldots \gtrdot \{0\}). \]

Fix a choice of isomorphism $V \cong \FM_q^n$.  Let $e_1, \ldots, e_n$ be the standard basis for $\FM_q^n$, let $G = GL(\FM_q^n)$ and $B < G$ be the subgroup of upper-triangular matrices.  Note that $G$ acts on $\Fl_n = \Fl(\FM_q^n)$ and that we may identify the $G$-modules $kG(\sum_{b \in B} b)$ and $\Fl_n$ by sending $(\sum_{b \in B} b)$ to the standard flag $(\FM_q^n \gtrdot \ldots \gtrdot \langle e_{1},e_{2}  \rangle \gtrdot \langle e_{1} \rangle \gtrdot \{0\} )$.  Under this isomorphism, Steinberg's element $\e$ gets sent to
\[ A_1 = \sum_{\sigma \in \mathfrak S_n} \sgn(\sigma) \sigma \cdot (\FM_q^n \gtrdot \ldots \gtrdot \langle e_{1},e_{2}  \rangle \gtrdot \langle e_{1} \rangle \gtrdot \{0\} ).\]
\[= \sum_{\sigma \in \mathfrak S_n} \sgn(\sigma)  (\FM_q^n \gtrdot \ldots \gtrdot \langle e_{\sigma(1)},e_{\sigma(2)}  \rangle \gtrdot \langle e_{\sigma(1)} \rangle \gtrdot \{0\} ).\]
More generally, for each $u \in U$, $u\e$ is sent to 
\[ A_u = \sum_{\sigma \in \mathfrak S_n} \sgn(\sigma) u \cdot (\FM_q^n \gtrdot \ldots \gtrdot \langle e_{\sigma(1)},e_{\sigma(2)}  \rangle \gtrdot \langle e_{\sigma(1)} \rangle \gtrdot \{0\} )\]
\[  = \sum_{\sigma \in \mathfrak S_n} \sgn(\sigma)  (\FM_q^n \gtrdot \ldots \gtrdot \langle u_{\sigma(1)},u_{\sigma(2)}  \rangle \gtrdot \langle u_{\sigma(1)} \rangle \gtrdot \{0\} ),\]
where $u_i = u e_i$ denotes the $i$th column of $u$ for $1\leq i \leq n$.  The classes $A_u$ are known as apartment classes and form a basis of $\tilde{H}_{n-2}(\cT(\FM_q^n);k)$.

Recall that $\Delta_{\{0\}}$  is isomorphic to the flag space $\Fs^+_{\{0\}} = \Fl^+_{\{0\}}/(FS)$.  Note also that the action of $G$ on $\FM_q^n$ induces a linear action on $\Delta_{\{0\}}^{\FM_q^n}$.   More generally, for any subspaces $V\subseteq W \subseteq \FM_q^n$, $GL(W/V)$ acts on the corresponding weight space $\Delta_{V}^{W}$ of $\Delta_V$.

Endowing $\Fl^+_{\{0\}}$ with the symmetric bilinear form $(~,~)$ for which the set of flags form an orthonormal basis, we identify $\Fl^+_{\{0\}}$ with its dual.  Let $\Fl_n \subset \Fl^+_{\{0\}}$ denote the subspace spanned by full flags.  The linear dual of the Steinberg module $\St_n^k(\FM_q)^* = (\Ker \partial)^*$ is isomorphic to the quotient $\Fl_n^*/\Im (\partial^*)$ of $\Fl_n^*$ by the image of the adjoint $\partial^*$.  Under the identification $\Fl_n^* =\Fl_n$, the image of $\partial^*$ is the $k$-submodule of $\Fl_n$ generated by the relation (FS).  We conclude that there is an isomorphism of $G$-modules $\Delta^{\FM_q^n}_{\{0\}} \cong \St_n^k(\FM_q)^*$ (and more generally of $GL(V/W)$-modules $\Delta^V_{W} \cong \St^k(V/W)^*$).

For each $u \in U$, let $F_u \in \Delta^{\FM_q^n}_{\{0\}}$ be the element represented by the full flag 
\[F_u = \sgn(w_0) u \cdot (\FM_q^n \gtrdot \ldots \gtrdot \langle e_{n-1},e_{n}  \rangle \gtrdot \langle e_{n} \rangle \gtrdot \{0\} )\] 
\[=  \sgn(w_0) (\FM_q^n \gtrdot \ldots \gtrdot \langle u_{n-1},u_{n}  \rangle \gtrdot \langle u_{n} \rangle \gtrdot \{0\} ),\]
or in other words, the flag whose $i$th part is the span of the last $i$ columns vectors of $u$.  For $u,u' \in U$, consider the pairing $(F_u, A_{u'})$.  Observe that the pairing 
\[ \left((\FM_q^n \gtrdot \ldots \gtrdot \langle u_{n-1},u_{n}  \rangle \gtrdot \langle u_{n} \rangle \gtrdot \{0\} ), (\FM_q^n \gtrdot \ldots \gtrdot \langle u'_{\sigma(1)},u'_{\sigma(2)}  \rangle \gtrdot \langle u'_{\sigma(1)} \rangle \gtrdot \{0\} )\right) \]
is equal to zero unless $\sigma = w_0$ and $u=u'$.  We conclude that $(F_u, A_{u'}) = \delta_{u,u'}$ and have shown: 

\begin{lemma}\label{lem-dualbasis}
Under the identification $\Delta^{\FM_q^n}_{\{0\}} \cong (\St^k(\FM_q^n))^*$ above, $\{F_u\}_{u\in U}$ is basis of $\Delta^{\FM_q^n}_{\{0\}}$ and dual to Steinberg's basis $\{u\e\}_{u\in U}$ of $\St^k(\FM_q^n)$.
\end{lemma}

\subsection{Dual filtrations} Assume $\ell = \mathrm{char}~k>0$.  For any subspaces $V\subseteq W \subseteq \FM_q^n$, the bilinear form on $\Delta^W_{V,k}$ is invariant under the natural $GL(W/V)$-action. Thus the Jantzen filtration on $\Delta_{V,k}$ restricts to a $GL(W/V)$-module filtration on $\Delta^{W}_{V,k} \cong \St^k(W/V)^*$.  To relate this filtration with Gow's filtration, we use the following lemma.

\begin{lemma}\label{lem-dualfilt}
Suppose that $\Lambda$ is a free $\Z$-module of finite rank $r$ and $\Lambda^* = \Hom(\Lambda,\Z)$ is the dual lattice.  Suppose that $\phi: \Lambda \to \Lambda^*$ and $\psi: \Lambda^* \to \Lambda$ are linear maps corresponding to symmetric bilinear forms $\mathbf B$ and $\mathbf D$ on $\Lambda$ and $\Lambda^*$ respectively such that $\psi \circ \phi = c \cdot \mathrm{Id}_\Lambda$ for some integer $c\in \Z$.  

For a fixed prime number $p$, let $\kappa = \nu_p(c)$ and
\[\Lambda = \Lambda^0 \supseteq \Lambda^1 \supseteq \Lambda^2 \supseteq \ldots \qquad
\mathrm{and} \qquad \Lambda^* = \Lambda^*(0) \supseteq \Lambda^*(1) \supseteq \Lambda^*(2) \supseteq \ldots\]
be the filtrations defined by
\[ \Lambda^i = \{ x \in \Lambda~|~\phi(x) \in p^i \Lambda^*\}, \qquad \Lambda^*(i) = \{ y \in \Lambda^*~|~\psi(y) \in p^i \Lambda\}.\]
Let $\overline{\Lambda} := \Lambda \otimes_\Z \Z/p$ and $\overline{\Lambda}^* := \Lambda^* \otimes_\Z \Z/p \cong \Hom_{\Z/p}(\overline{\Lambda}, \Z/p)$ with the filtrations
\[\overline\Lambda = \overline\Lambda^0 \supseteq \overline\Lambda^1 \supseteq \overline\Lambda^2 \supseteq \ldots\]
and
\[\overline{\Lambda}^* = \overline{\Lambda}^*(0) \supseteq \overline{\Lambda}^*(1) \supseteq \overline{\Lambda}^*(2) \supseteq \ldots,\]
where $\overline\Lambda^i$ (resp. $\overline{\Lambda}^*(i)$) is the image in $\overline\Lambda$ (resp. $\overline{\Lambda}^*$) of $\Lambda^i \otimes \Z/p$ (resp. ${\Lambda}^*(i) \otimes \Z/p$).
Then the filtrations of $\overline{\Lambda}$ and $\overline{\Lambda}^*$ are dual in the sense that for any $j$
\[ (\overline\Lambda^j)^\perp = \overline{\Lambda}^*(\kappa+1-j). \]
\end{lemma}

\begin{proof}
As $\Z$ is a principal ideal domain, we can choose bases $\{m_i\}_{1\leq i \leq r}$ and $\{m'_i\}_{1\leq i \leq r}$ of $\Lambda$ such that $\mathbf B(m_i,m_j') = a_i \delta_{i,j}$ for all $1\leq i,j\leq r$.  Let $\{m^*_i\}_{1\leq i \leq r}$ and $\{{m'_i}^*\}_{1\leq i \leq r}$ be the corresponding dual bases of $\Lambda^*$. By our definitions, then $\phi(m_i') = a_i m_i^*$ and $\phi(m_i) = a_i {m'_i}^{*}$.

Using the assumption $\psi \circ \phi = c \cdot \mathrm{Id}$, \[cm'_i=\psi(\phi(m'_i)) = \psi(a_i {m_i}^*) = a_i \psi ({m_i}^*),\] 
it follows that $\psi({m_i}^*) = \frac{c}{a_i} m'_i$ and $\mathbf D(m_i^*, {m'_j}^*) = \frac{c}{a_i} \delta_{i,j}$.

Now the set of $m_i \otimes 1$ with $\nu_p(a_i) \geq j$ is a basis for $\overline{\Lambda}^j$ (as is the set of $m_i' \otimes 1$ with $\nu_p(a_i) \geq j$).  In the same way, the set of $m_i^* \otimes 1$ with $\nu_p(c/a_i) \geq j$ is a basis for $\overline{\Lambda}^*(j)$.

Now consider the orthogonal subspace $(\overline{\Lambda}^j)^\perp \subseteq \overline{\Lambda}^*$. As $\{m_i \otimes 1\}_{1\leq i \leq r}$ and $\{m^*_i \otimes 1\}_{1\leq i \leq r}$ are dual bases,
\[ (\overline{\Lambda}^j)^\perp = \{m_i \otimes 1~|~\nu_p(a_i) \geq j \}^\perp = \mathrm{span}\{m_i^* \otimes 1 ~|~ \nu_p(a_i) < j \}. \]
Note that $\nu_p(a_i) < j $ if and only if $\nu_p(c/a_i) = \kappa - \nu_p(a_i) \geq \kappa +1 -j$.  Thus,
\[ (\overline{\Lambda}^j)^\perp = \mathrm{span}\{m_i^* \otimes 1 ~|~ \nu_p(c/a_i) \geq \kappa +1 -j \} = \overline{\Lambda}^*(\kappa+1-j), \]
as we wished to show.
\end{proof}

We note that to check the condition of the lemma, it is enough to produce a basis $\{m_i\}_{1\leq i \leq r}$ of $\Lambda$ such that if $\{m_i^*\}_{1\leq i \leq r}$ is the dual basis of $\Lambda^*$, then the Gram matrix of $\mathbf B$ with respect to $\{m_i\}_{1\leq i \leq r}$ times the Gram matrix of $\mathbf D$ with respect to $\{m_i^*\}_{1\leq i \leq r}$ is an integral multiple of the identity matrix.

\begin{theorem}\label{thm-gowfilt}
The restriction of the filtration \[\Delta_{\{0\},k} = (\Delta_{\{0\},k})^0 \supset (\Delta_{\{0\},k})^1 \supseteq \ldots \supseteq (\Delta_{\{0\},k})^N=0,\]
to the $\FM_q^n$-weight space $\Delta^{\FM_q^n}_{\{0\},k} \cong \St^k(\FM_q^n)^*$ is dual to Gow's Jantzen filtration of $\St(\FM_q^n)$,
\[ \St^k(\FM_q^n) = I(0) \supseteq I(1) \supseteq I(2) \supseteq \ldots , \]
in the sense that $I(i)^\perp = (\Delta^{\FM_q^n}_{\{0\},k})^{\kappa+1-i}$, where $\kappa:= \nu_p(GL(\FM_q^n)/B)$.  
\end{theorem}

\begin{proof}  By Lemma~\ref{lem-dualbasis} and the discussion above, to apply Lemma~\ref{lem-dualfilt} it suffices to show that the product of the Gram matrix $\mathbf B$ for $\langle~,~\rangle_{\{0\}}^{\FM_q^n}$ with respect to the basis $\{F_u\}_{u\in U}$ times the Gram matrix $\mathbf G$ for Gow's symmetric form on $\St^k(\FM_q^n)$ with respect to the basis $\{u\e~|~u\in U\}$ is $|GL(\FM_q^n)/B|$.

For $g \in G$, let $F(g)$ denote the flag
\[ F(g) = (\FM_q^n \gtrdot \ldots \gtrdot \langle ge_{n-1},ge_{n}  \rangle \gtrdot \langle ge_{n} \rangle \gtrdot \{0\} ) \]

By~\cite{BV}, the inverse matrix of $\mathbf B$ is the Gram matrix for the pairing on $\St_n^k(\FM_q)$ defined by
\[ C(u\e,u'\e) = C(\e,u^{-1}u'\e ) = \sum_{\substack{\sigma, \tau \in \mathfrak S_n, F \\  F(\sigma) = F(u^{-1}u' \tau) = F}} \frac{\sgn(\sigma) \sgn(\tau)}{a(F)},\]
where the sum is over all flags $F$ and pairs of permutations $\sigma,\tau \in \mathfrak S_n$ such that $F(\sigma) = F(u^{-1}u' \tau) = F$.  Here we have set Brylawski--Varchenko's weight function $a$ on the set of elements of $PG(n-1,q)$ to be identically 1, so that $a(F)$ is the product of the number of elements in each flat of $F$, which is $\prod_{i=1}^{n} \frac{q^i-1}{q-1} = |GL_n(\FM_q)/B(\FM_q)|$.

Notice that $F(\sigma)$ is the flag generated by a permutation of the standard basis and as $u^{-1}u'$ is upper triangular, the equality $F(\sigma) = F(u^{-1}u' \tau)$ can hold only if $\sigma = \tau$.  Thus the sign of every term is positive and each term contributes $|GL_n(\FM_q)/B(\FM_q)|^{-1}$.  The number of terms is the number of $\sigma \in \mathfrak S_n$ such that 
$F(\sigma) = F(u^{-1}u' \sigma)$.  This condition holds if and only if $\sigma^{-1} u^{-1}u' \sigma \in U$. We conclude by equation (\ref{eq-f}) that
\[ C(u\e,u'\e) = \frac{c_W(u^{-1}u')}{|GL_n(\FM_q)/B(\FM_q)|} = \frac{f(u\e,u'\e)}{|GL_n(\FM_q)/B(\FM_q)|}.\]
Thus the inverse matrix of $B$ is the Gram matrix of $f$ divided by $|GL_n(\FM_q)/B(\FM_q)|$.  We may therefore apply Lemma~\ref{lem-dualfilt} and the theorem follows.
\end{proof}

\begin{corollary} The $\FM_q^n$-weight space of the $i$th successive quotient $(\Delta_{\{0\},k})^{i}/(\Delta_{\{0\},k})^{i+1}$ of the Jantzen filtration of $\Delta_{\{0\},k}$ is isomorphic as a $GL_n(\FM_q)$-module to the successive quotient $M(\kappa-i)$ of Gow's Jantzen filtration of $\St^k(\FM_q^n)$, where $\kappa$ is the $\ell$-valuation of the number of full flags in $\FM_q^n$.

In particular, the $\FM_q^n$-weight space of the simple module $L_{\{0\},k} = (\Delta_{\{0\},k})^{0}/(\Delta_{\{0\},k})^{1}$ is isomorphic to the socle $\mathrm{soc}(\St^k(\FM_q^n))$ of $\St^k(\FM_q^n)$.
\end{corollary}

\begin{proof}  Applying the previous theorem,
\[(\Delta^{\FM_q^n}_{\{0\},k})^{i}/(\Delta^{\FM_q^n}_{\{0\},k})^{i+1} \cong I(\kappa+1-i)^\perp/I(\kappa-i)^\perp  
\cong (I(\kappa-i)/I(\kappa-i+1))^* = M(\kappa-i)^*.\]
By~\cite[Theorem 4.3]{Gow}, $M(\kappa-i)$ is self-dual.

In particular $L_{\{0\},k}^{\FM_q^n} \cong (\Delta^{\FM_q^n}_{\{0\},k})^{0}/(\Delta^{\FM_q^n}_{\{0\},k})^{1} \cong M(\kappa)$, which is the socle of $\St^k(\FM_q^n)$ by~\cite[Theorem 4.5]{Gow}.
\end{proof}

More generally, for any subspaces $V \subseteq W \subseteq \FM_q^n$,  note that the weight space $\Delta_V^W$ is naturally isomorphic to $\Delta_{\{0\}}^{W/V}$ and thus is naturally a $GL(W/V)$-representation isomorphic to the dual Steinberg module $\St^k(W/V)^*$.  By equation~(\ref{eq-form}), under this isomorphism the restriction of the form $\langle~,~\rangle_V$ on $\Delta_V$ to $\Delta_V^W$ is equal to $q^{\dim V \cdot \dim (W/V)}$ times the form $\langle~,~\rangle_{\{0\}}$ on $\Delta^{W/V}_{\{0\}}$.  Combining this observation with the theorem above, we can obtain a complete description of the simple $S(PG(n-1,q))$-modules.

\begin{corollary}\label{cor-weight} 
Recall that $\ell$ denotes the characteristic of $k$.  
\begin{enumerate}
\item If $\ell=p$ and $q=p^s$, then $\Delta_{\{0\},k}=L_{\{0\},k}$ and for any $V \neq \{0\}$ the filtration on $\Delta_{V,k}$ is given by 
\[ (\Delta_{V,k})^{i} = \bigoplus_{\substack{V \subseteq W \\ s(\dim V)(\dim W/V) \geq i}} (\Delta^W_{V,k}). \]
In particular, $L_{V,k} \cong \Delta^V_{V,k}$ is one-dimensional for each $V \neq \{0\}$.

\item If $\ell$ does not divide $q$, then for $V \subseteq W$, the $W$-weight space of the $i$th successive quotient $(\Delta_{V,k})^{i}/(\Delta_{V,k})^{i+1}$ of the Jantzen filtration of $\Delta_{V}$ is isomorphic as a $GL(W/V)$-module to the successive quotient $M(\kappa-i)$ of Gow's Jantzen filtration of $\St^k(W/V)$, where $\kappa$ is the $\ell$-valuation of the number of full flags in $W/V$.

In particular, the $W$-weight space of the simple module $L_{V,k} = (\Delta_{V,k})^{0}/(\Delta_{V,k})^{1}$ is isomorphic to the socle $\mathrm{soc}(\St^k(W/V))$ of $\St^k(W/V)$.
\end{enumerate}
\end{corollary}

\begin{remark}
We direct to the reader to~\cite[Example 4.9, Remarks 4.10 and 4.11]{Geck} for more information on the structure of $\St^k(\FM_q^n)$.  In particular, Geck computes the composition length and gives a conjectural description of the partitions corresponding to composition factors.
\end{remark}

It would be interesting to compute the composition series multiplicities $[\Delta_{V,k}: L_{W,k}]$.  By the highest weight structure, this multiplicity vanishes unless $V \subseteq W$.  Note also that by symmetry, for any element $g \in GL_n(\FM_q)$, there is an equality
\[ [\Delta_{V,k}: L_{W,k}] = [\Delta_{gV,k}: L_{gW,k}].\]
As $GL_n(\FM_q)$ acts transitively on the set of nested pairs of subspaces of fixed dimensions, for any $V\subseteq W$, the multiplicity $[\Delta_{V,k}: L_{W,k}]$ only depends on the dimensions of $V$ and $W$.  By the results above, computing this number is basically equivalent to computing the dimension of the socle of the Steinberg module.  We conclude this section by working out these multiplicities in the case $\ell ~|~ q+1$, where we can apply Khammash's result that the socle of the Steinberg module is trivial (see Remark~\ref{rem-kham}).

\begin{proposition}
Suppose that $\ell~|~q+1$ and $V\subseteq W$ are subspaces of $\FM_q^n$.  Let $i$ be the codimension of $V$ in $W$.  Then the composition series multiplicity
\[ [ \Delta_{V,k}:L_{W,k}] = \begin{cases}
(q^{2i-1} - 1)(q^{2i-1} - q^2) \ldots (q^{2i-1} - q^{2i-2}) & \mathrm{if}~i~\mathrm{is~even} \\
0 &  \mathrm{if}~i~\mathrm{is~odd}.
\end{cases}\]
\end{proposition}

\begin{proof}
In the Grothendieck group of $S_k(PG(n-1,q))$-modules, we have the equality
\[ [\Delta_{V,k}] = \sum_{V \subseteq W} [ \Delta_{V,k}:L_{W,k}]\cdot [L_{W,k}].\]
For any subspace $U \subseteq \FM_q^n$ such that $V \subseteq U$ of codimension $j$, considering the dimension of the $U$-weight space on each side of the equation gives:
\[ q^{\binom{j}{2}} = \sum_{W \supseteq V} [\Delta_{V,k}: L_{W,k}] \dim(L_{W,k}^U). \]
By Corollary~\ref{cor-weight} and Remark~\ref{rem-kham}, under our assumption $\ell ~|~ q+1$, $\dim(L_{W,k}^U)$ is equal to $1$ if $W \subseteq U$ and $0$ otherwise, so the formula above simplifies to:
\[ q^{\binom{j}{2}} = \sum_{U\supseteq W \supseteq V} [\Delta_{V,k}: L_{W,k}], \]
where $V\subseteq U$ is codimension $j$ and the sum is over all subspaces $W$ that contain $V$ and are contained in $U$.  Recall that by symmetry the multiplicity $[\Delta_{V,k}: L_{W,k}]$ only depends on the dimensions of $V$ and $W$, so the formula can be rewritten:
\[ q^{\binom{j}{2}} = \sum_{i=0}^{j} \binom{j}{i}_q [\Delta_{V,k}: L_{W^i,k}],\]
where for each $i=0,\ldots,j$, $W^i$ is a fixed choice of subspace of dimension $\dim V +i$ such that $V \subseteq W^i \subseteq \FM_q^n$.
We see from this system of equations that the numbers $[\Delta_{V,k}, L_{W,k}]$ can be computed inductively and only depend on the codimension of $V \subseteq W$.  It remains to check that the expression in the statement of the proposition is a solution to the system of equations.

\begin{lemma}
For any natural number $j$, 
\begin{equation}\label{eq-binom} q^{\binom{j}{2}} = \sum_{i=0}^{\lfloor j/2\rfloor} \binom{j}{2i}_q \cdot \prod_{m = 0}^{i-1} (q^{2i-1} - q^{2m}).\end{equation}
\end{lemma}

Note that when $j=0$ or $1$, the left hand side is $q^{\binom{j}{2}}=q^0=1$ and the right hand side reduces a single term that is also equal to $1$ (as the product is empty).  We proceed by induction on $j$, assuming that $j \geq 2$ and that the lemma is true in the $(j-1)$ and $(j-2)$-cases.

Using the equality $\binom{j}{2i}_q = \binom{j-1}{2i}_q + \binom{j-1}{2i-1}_q q^{j-2i}$ we have
\[ \sum_{i=0}^{\lfloor j/2 \rfloor} \binom{j}{2i}_q \cdot \prod_{m = 0}^{i-1} (q^{2i-1} - q^{2m}) = \sum_{i=0}^{\lfloor j/2 \rfloor} \left(\binom{j-1}{2i}_q + \binom{j-1}{2i-1}_q q^{j-2i}\right) \cdot \prod_{m = 0}^{i-1} (q^{2i-1} - q^{2m})\]
\[ = \sum_{i=0}^{\lfloor j/2 \rfloor} \binom{j-1}{2i}_q  \cdot \prod_{m = 0}^{i-1} (q^{2i-1} - q^{2m})  + \sum_{i=1}^{\lfloor j/2 \rfloor} \binom{j-1}{2i-1}_q q^{j-2i} \cdot \prod_{m = 0}^{i-1} (q^{2i-1} - q^{2m}).\]
We have begun the second sum at $i=1$ because the $q$-binomial coefficient $\binom{j-1}{-1}_q$ vanishes. 

Note that if $j$ is even and $i = j/2$, then $\binom{j-1}{2i}_q = \binom{j-1}{j}_q = 0$, while if $j$ is odd, then $\lfloor j/2 \rfloor = \lfloor(j-1)/2 \rfloor$.  In either case, the first sum can be replaced by $q^{\binom{j-1}{2}}$ by the induction hypothesis for $j-1$.  On the other hand, $\binom{j-1}{2i-1}_q = \binom{j-2}{2i-2}_q \cdot \frac{q^{j-1}-1}{q^{2i-1}-1}$ and \[\prod_{m = 0}^{i-1} (q^{2i-1} - q^{2m}) = q^{2i-2}\cdot (q^{2i-1}-1) \cdot \prod_{m = 0}^{i-2} (q^{2i-1} - q^{2m}).\]
Making these three substitutions to the expression above we get:
\[ = q^{\binom{j-1}{2}} + \sum_{i=1}^{\lfloor j/2 \rfloor} \binom{j-2}{2i-2}_q \cdot \frac{q^{j-1}-1}{q^{2i-1}-1} q^{j-2i} q^{2i-2}  (q^{2i-1}-1) \cdot \prod_{m = 0}^{i-1} (q^{2i-1} - q^{2m}).\]
After cancellations this gives:
\[ = q^{\binom{j-1}{2}} + \sum_{i=1}^{\lfloor j/2 \rfloor} \binom{j-2}{2i-2}_q \cdot (q^{j-1}-1) q^{j-2} \cdot \prod_{m = 0}^{i-2} (q^{2i-3} - q^{2m}).\]
\[ = q^{\binom{j-1}{2}} + (q^{j-1}-1) q^{j-2} \cdot \left(\sum_{i=0}^{\lfloor (j-2)/2 \rfloor} \binom{j-2}{2i}_q  \cdot \prod_{m = 0}^{i-1} (q^{2i-1} - q^{2m})\right).\]
We conclude by noting that the sum is the $(j-2)$-case of our expression and so can be replaced by $q^{\binom{j-2}{2}}$:
\[ = q^{\binom{j-1}{2}} + (q^{j-1}-1) q^{j-2} q^{\binom{j-2}{2}} = q^{\binom{j-1}{2}} + (q^{j-1}-1) q^{{j-2}+\binom{j-2}{2}}\]
\[ = q^{\binom{j-1}{2}} + (q^{j-1}-1) q^{\binom{j-1}{2}} = q^{\binom{j-1}{2}} + q^{\binom{j}{2}} - q^{\binom{j-1}{2}} =q^{\binom{j}{2}}.\]
This completes our proof of the lemma and thus the proposition as well.
\end{proof}

It could be interesting to have a geometric interpretation of equation~\ref{eq-binom}.

\section{Complete graphic Schur algebras}\label{sec-Kn}

 Let $n$ be a natural number.  Let $\Pi_n$ denote the set of all partitions of the set $[n]=\{1,\ldots,n\}$.  Here by a set partition of $[n]$ we mean a collection $\sigma$ of pairwise disjoint nonempty subsets $B_1,\ldots, B_l$ of $[n]$ for which $\bigcup_i B_i = [n]$.  We say that each $B_i$ is a block of $\sigma$ and write $|\sigma|$ to denote the number of blocks of $\sigma$.  We can write a set partition $\sigma$ using the notation $B_1 / B_2 / \ldots / B_l$.  We partially order $\Pi_n$ by refinement, writing $\pi \leq \sigma$ if every block of $\pi$ is contained in a block of $\sigma$.  We say $\pi$ is covered by $\sigma$ and write $\pi \lessdot \sigma$ if $\sigma$ is obtained from $\pi$ by merging two blocks into one.
 
Let $K_n$ denote the complete graph on $[n]$.  We recall:

\begin{definition}
Let $E$ be the set of edges of $K_n$.  The \textit{graphic matroid $M(K_n)$ for the complete graph} is the matroid with underlying set $E$ and $\Lc = \Pi_n$, where we view a set partition $\sigma$ of $[n]$ as corresponding to the set of all edges of $K_n$ that do not connect distinct blocks of $\sigma$.  The rank $\rho(\sigma)$ of a set partition $\sigma = B_1 / B_2 / \ldots / B_l$ is equal to $n-l$.
\end{definition}

A flag path in $\Pi_n$ is then a sequence of set partitions $(\sigma_1, \ldots, \sigma_l)$ such that for each $i = 1, \ldots, l-1$, either $\sigma_i \lessdot \sigma_{i+1}$ or $\sigma_i \gtrdot \sigma_{i+1}$.  

\begin{remark}(On diagrammatics) We can express flag paths in set partitions diagrammatically in terms of colored strands as follows.  We denote a single set partition  $\sigma= B_1 / B_2 / \ldots / B_k$ as a collection of $k$ clusters of colored dots arranged on a horizontal line, where each cluster corresponds to a block $B_i$ and the colored dots label the elements (or nonempty sets of elements) of $B_i$.  For example, we might express the partition $\sigma = 12/3$ as: 

\[
\begin{tikzpicture}[baseline = -3]
	\draw[blue,-,ultra thick] (-1,-.1) to (-1,.1);
	\draw[red,-,ultra thick] (-1.05,-.1) to (-1.05,.1);
	\draw[green,-,ultra thick] (0,-.1) to (0,.1);
	\draw[densely dotted] (-1.2,0) -- (.3,0);
	\node at (-1,-.4) {$\{1,2\}$};
        \node at (0,-.4) {$\{3\}$};
\end{tikzpicture}, \qquad
\begin{tikzpicture}[baseline = -3]
	\draw[blue,-,ultra thick] (0,-.1) to (0,.1);
	\draw[red,-,ultra thick] (0.05,-.1) to (0.05,.1);
	\draw[green,-,ultra thick] (-1,-.1) to (-1,.1);
	\draw[densely dotted] (-1.2,0) -- (.3,0);
	\node at (-1,-.4) {$\{3\}$};
        \node at (0,-.4) {$\{1,2\}$};
\end{tikzpicture},
\qquad \mathrm{or} \qquad
\begin{tikzpicture}[baseline = -3]
	\draw[violet,-,ultra thick] (0,-.1) to (0,.1);
	\draw[green,-,ultra thick] (-1,-.1) to (-1,.1);
	\draw[densely dotted] (-1.2,0) -- (.3,0);
	\node at (-1,-.4) {$\{3\}$};
        \node at (0,-.4) {$\{1,2\}$};
\end{tikzpicture},
\]
where blue denotes 1, red 2, green 3 and purple $\{1,2\}$.
Note that the order and position of the clusters on the line does not matter -- only the content of the clusters.

We may then view a flag path as a sequence of such cluster diagrams, where at each step either two clusters merge into one or one cluster splits into two.  To express a flag path diagrammatically we draw the sequence of cluster diagrams on stacked horizontal lines, from bottom to top, and connect the dots by colored strands.  For example, we could represent the flag path $(12/3, 1/2/3, 1/23, 123)$ by the diagram:
\[\begin{tikzpicture}[baseline = -3]
	\draw[blue,-,ultra thick] (.55,-1) to (.55,1);
	\draw[blue,-,ultra thick] (.55,1) to (0.05,1.5);
	\draw[blue,-,ultra thick] (.05,2) to (0.05,1.5);
	\draw[green,-,ultra thick] (-0.48,-.5) to (0.5,.5);
	\draw[green,-,ultra thick] (-0.48,-.5) to (-.48,-1);
	\draw[green,-,ultra thick] (.5,1) to (0.5,.5);
	\draw[green,-,ultra thick] (0,1.5) to (0.5,1);
	\draw[green,-,ultra thick] (0,1.5) to (0,2);
	\draw[red,-,ultra thick] (-.52,-1) to (-.52,1);
	\draw[red,-,ultra thick] (-.52,1) to (-.05,1.5);
	\draw[red,-,ultra thick] (-.05,2) to (-.05,1.5);
	
	\draw[densely dotted] (-.8,-1) -- (.8,-1);
	\draw[densely dotted] (-.8,0) -- (.8,0);
	\draw[densely dotted] (-.8,1) -- (.8,1);
	\draw[densely dotted] (-.8,2) -- (.8,2);
\end{tikzpicture},
\]
where red represents 1, green represents 2 and blue represents 3.

In expressing the relations below diagrammatically, we omit drawing blocks that do not merge or split.

One could equally well draw the diagrams with just labels and not the colors, but the colors may make it easier to visualize.
\end{remark}

Applying our definition of $S(M)$ to $M=M(K_n)$, we obtain:

\begin{definition} The complete graphic matroidal Schur algebra $S(M(K_n))$ is the quotient of $\FlP(M(K_n))$ by the following relations:
\begin{itemize}
\item[(FS')] Let $\sigma \in \Pi_n$ be a partition that contains distinct blocks $B,B',B''$.  Let $\tau$ (resp. $\tau'$ and $\tau''$) be the partition obtained from $\sigma$ by merging the blocks $B,B'$ (resp, $B',B''$ and $B'',B$) into a single block.  Let $\pi$ be the partition obtained from $\sigma$ by merging $B,B'$ and $B''$ into a single block.  Then:
\[(\sigma \lessdot  \tau \lessdot \pi) + (\sigma\lessdot \tau'\lessdot \pi) + (\sigma\lessdot \tau''\lessdot \pi) = 0 = (\pi\gtrdot \tau\gtrdot \sigma) + (\pi\gtrdot \tau'\gtrdot \sigma) + (\pi\gtrdot \tau''\gtrdot \sigma).\]
Alternatively, we may draw these relations diagrammatically as:

\[
\begin{tikzpicture}[scale=0.50,baseline = -3]
	\draw[blue,-,ultra thick] (0.15,.5) to (0.15,1);
	\draw[green,-,ultra thick] (0.06,.5) to (0.06,1);
	\draw[red,-,ultra thick] (-.05,.5) to (-.05,1);
	\draw[blue,-,ultra thick] (1,-1) to (0.15,.5);
	\draw[green,-,ultra thick] (0,-1) to (-0.5,-.5);
	\draw[green,-,ultra thick] (-0.5,-0.5) to (0.06,.5);
	\draw[red,-,ultra thick] (-1,-1) to (-0.6,-.5);
	\draw[red,-,ultra thick] (-0.6,-.5) to (-.05,.5);
	\draw[densely dotted] (-1.2,-1) -- (1.2,-1);
	\draw[densely dotted] (-1.2,0) -- (1.2,0);
	\draw[densely dotted] (-1.2,1) -- (1.2,1);
		\node at (-1,-1.3) {$\scriptstyle B$};
		\node at (0,-1.3) {$\scriptstyle B'$};
		\node at (1,-1.3) {$\scriptstyle B''$};
\end{tikzpicture}
+
\begin{tikzpicture}[scale=0.50,baseline = -3]
	\draw[red,-,ultra thick] (0.15,.5) to (0.15,1);
	\draw[blue,-,ultra thick] (0.06,.5) to (0.06,1);
	\draw[green,-,ultra thick] (-.05,.5) to (-.05,1);
	\draw[red,-,ultra thick] (1,-1) to (0.15,.5);
	\draw[blue,-,ultra thick] (0,-1) to (-0.5,-.5);
	\draw[blue,-,ultra thick] (-0.5,-0.5) to (0.06,.5);
	\draw[green,-,ultra thick] (-1,-1) to (-0.6,-.5);
	\draw[green,-,ultra thick] (-0.6,-.5) to (-.05,.5);
	\draw[densely dotted] (-1.2,-1) -- (1.2,-1);
	\draw[densely dotted] (-1.2,0) -- (1.2,0);
	\draw[densely dotted] (-1.2,1) -- (1.2,1);
		\node at (-1,-1.3) {$\scriptstyle B'$};
		\node at (0,-1.3) {$\scriptstyle B''$};
		\node at (1,-1.3) {$\scriptstyle B$};
\end{tikzpicture}
+
\begin{tikzpicture}[scale=0.50,baseline = -3]
	\draw[green,-,ultra thick] (0.15,.5) to (0.15,1);
	\draw[red,-,ultra thick] (0.06,.5) to (0.06,1);
	\draw[blue,-,ultra thick] (-.05,.5) to (-.05,1);
	\draw[green,-,ultra thick] (1,-1) to (0.15,.5);
	\draw[red,-,ultra thick] (0,-1) to (-0.5,-.5);
	\draw[red,-,ultra thick] (-0.5,-0.5) to (0.06,.5);
	\draw[blue,-,ultra thick] (-1,-1) to (-0.6,-.5);
	\draw[blue,-,ultra thick] (-0.6,-.5) to (-.05,.5);
	\draw[densely dotted] (-1.2,-1) -- (1.2,-1);
	\draw[densely dotted] (-1.2,0) -- (1.2,0);
	\draw[densely dotted] (-1.2,1) -- (1.2,1);
		\node at (-1,-1.3) {$\scriptstyle B''$};
		\node at (0,-1.3) {$\scriptstyle B$};
		\node at (1,-1.3) {$\scriptstyle B'$};
\end{tikzpicture}
=0 =
\begin{tikzpicture}[scale=0.50,baseline = -3]
\begin{scope}[yscale=-1,xscale=1]
	\draw[blue,-,ultra thick] (0.15,.5) to (0.15,1);
	\draw[green,-,ultra thick] (0.06,.5) to (0.06,1);
	\draw[red,-,ultra thick] (-.05,.5) to (-.05,1);
	\draw[blue,-,ultra thick] (1,-1) to (0.15,.5);
	\draw[green,-,ultra thick] (0,-1) to (-0.5,-.5);
	\draw[green,-,ultra thick] (-0.5,-0.5) to (0.06,.5);
	\draw[red,-,ultra thick] (-1,-1) to (-0.6,-.5);
	\draw[red,-,ultra thick] (-0.6,-.5) to (-.05,.5);
	\draw[densely dotted] (-1.2,-1) -- (1.2,-1);
	\draw[densely dotted] (-1.2,0) -- (1.2,0);
	\draw[densely dotted] (-1.2,1) -- (1.2,1);
\end{scope}
		\node at (-1,1.3) {$\scriptstyle B$};
		\node at (0,1.3) {$\scriptstyle B'$};
		\node at (1,1.3) {$\scriptstyle B''$};
\end{tikzpicture}
+
\begin{tikzpicture}[scale=0.50,baseline = -3]
\begin{scope}[yscale=-1,xscale=1]
	\draw[red,-,ultra thick] (0.15,.5) to (0.15,1);
	\draw[blue,-,ultra thick] (0.06,.5) to (0.06,1);
	\draw[green,-,ultra thick] (-.05,.5) to (-.05,1);
	\draw[red,-,ultra thick] (1,-1) to (0.15,.5);
	\draw[blue,-,ultra thick] (0,-1) to (-0.5,-.5);
	\draw[blue,-,ultra thick] (-0.5,-0.5) to (0.06,.5);
	\draw[green,-,ultra thick] (-1,-1) to (-0.6,-.5);
	\draw[green,-,ultra thick] (-0.6,-.5) to (-.05,.5);
	\draw[densely dotted] (-1.2,-1) -- (1.2,-1);
	\draw[densely dotted] (-1.2,0) -- (1.2,0);
	\draw[densely dotted] (-1.2,1) -- (1.2,1);
\end{scope}
		\node at (-1,1.3) {$\scriptstyle B'$};
		\node at (0,1.3) {$\scriptstyle B''$};
		\node at (1,1.3) {$\scriptstyle B$};
\end{tikzpicture}
+
\begin{tikzpicture}[scale=0.50,baseline = -3]
\begin{scope}[yscale=-1,xscale=1]
	\draw[green,-,ultra thick] (0.15,.5) to (0.15,1);
	\draw[red,-,ultra thick] (0.06,.5) to (0.06,1);
	\draw[blue,-,ultra thick] (-.05,.5) to (-.05,1);
	\draw[green,-,ultra thick] (1,-1) to (0.15,.5);
	\draw[red,-,ultra thick] (0,-1) to (-0.5,-.5);
	\draw[red,-,ultra thick] (-0.5,-0.5) to (0.06,.5);
	\draw[blue,-,ultra thick] (-1,-1) to (-0.6,-.5);
	\draw[blue,-,ultra thick] (-0.6,-.5) to (-.05,.5);
	\draw[densely dotted] (-1.2,-1) -- (1.2,-1);
	\draw[densely dotted] (-1.2,0) -- (1.2,0);
	\draw[densely dotted] (-1.2,1) -- (1.2,1);
\end{scope}
		\node at (-1,1.3) {$\scriptstyle B''$};
		\node at (0,1.3) {$\scriptstyle B$};
		\node at (1,1.3) {$\scriptstyle B'$};
\end{tikzpicture}
.
\]

\item[(FS'')] Suppose $\sigma \in \Pi_n$ contains distinct blocks $B,B',C,C'$.  Let $\mu$ (resp. $\nu$) be the partition obtained from $\sigma$ by merging the blocks $B,B'$  (resp. $C,C'$).  Let $\pi$ be the join of $\mu$ and $\nu$ (so the partition obtained from $\sigma$ by merging the blocks $B,B'$ into one block and $C,C'$ into another).  Then:
\[ (\sigma \lessdot \mu \lessdot \pi) + (\sigma \lessdot \nu \lessdot \pi) = 0 = (\pi \gtrdot \mu \gtrdot\sigma) + (\pi \gtrdot \nu \gtrdot \sigma).\]
This can be expressed diagrammatically as:
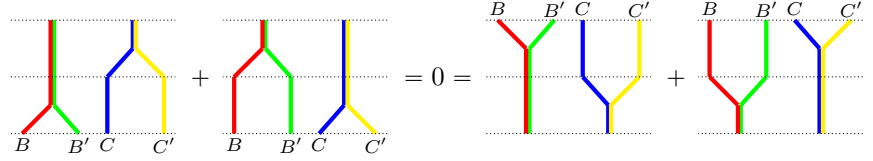
\[
\begin{tikzpicture}[scale=0.75,baseline = -3]
	\draw[blue,-,ultra thick] (.95,1) to (.95,.5);
	\draw[blue,-,ultra thick] (.5,0) to (.95,.5);
	\draw[blue,-,ultra thick] (.5,0) to (.5,-1);
	\draw[yellow,-,ultra thick] (1,1) to (1,.5);
	\draw[yellow,-,ultra thick] (1.5,0) to (1,.5);
	\draw[yellow,-,ultra thick] (1.5,0) to (1.5,-1);
	\draw[green,-,ultra thick] (0,-1) to (-0.44,-0.5);
	\draw[green,-,ultra thick] (-0.44,-0.5) to (-0.44,1);
	\draw[red,-,ultra thick] (-1,-1) to (-.5,-.5);
	\draw[red,-,ultra thick] (-.5,-.5) to (-.5,1);
	\draw[densely dotted] (-1.2,-1) -- (1.7,-1);
	\draw[densely dotted] (-1.2,0) -- (1.7,0);
	\draw[densely dotted] (-1.2,1) -- (1.7,1);
		\node at (-1,-1.2) {$\scriptstyle B$};
		\node at (0,-1.2) {$\scriptstyle B'$};
		\node at (.5,-1.2) {$\scriptstyle C$};
		\node at (1.5,-1.2) {$\scriptstyle C'$};
\end{tikzpicture}
+
\begin{tikzpicture}[scale=0.75,baseline = -3]
\begin{scope}[yscale=1,xscale=-1]
	\draw[green,-,ultra thick] (.95,1) to (.95,.5);
	\draw[green,-,ultra thick] (.5,0) to (.95,.5);
	\draw[green,-,ultra thick] (.5,0) to (.5,-1);
	\draw[red,-,ultra thick] (1,1) to (1,.5);
	\draw[red,-,ultra thick] (1.5,0) to (1,.5);
	\draw[red,-,ultra thick] (1.5,0) to (1.5,-1);
	\draw[blue,-,ultra thick] (0,-1) to (-0.44,-0.5);
	\draw[blue,-,ultra thick] (-0.44,-0.5) to (-0.44,1);
	\draw[yellow,-,ultra thick] (-1,-1) to (-.5,-.5);
	\draw[yellow,-,ultra thick] (-.5,-.5) to (-.5,1);
	\draw[densely dotted] (-1.2,-1) -- (1.7,-1);
	\draw[densely dotted] (-1.2,0) -- (1.7,0);
	\draw[densely dotted] (-1.2,1) -- (1.7,1);
\end{scope}
		\node at (1,-1.2) {$\scriptstyle C'$};
		\node at (0,-1.2) {$\scriptstyle C$};
		\node at (-.5,-1.2) {$\scriptstyle B'$};
		\node at (-1.5,-1.2) {$\scriptstyle B$};
\end{tikzpicture} = 0 =
\begin{tikzpicture}[scale=0.75,baseline = -3]
\begin{scope}[yscale=-1,xscale=1]
	\draw[blue,-,ultra thick] (.95,1) to (.95,.5);
	\draw[blue,-,ultra thick] (.5,0) to (.95,.5);
	\draw[blue,-,ultra thick] (.5,0) to (.5,-1);
	\draw[yellow,-,ultra thick] (1,1) to (1,.5);
	\draw[yellow,-,ultra thick] (1.5,0) to (1,.5);
	\draw[yellow,-,ultra thick] (1.5,0) to (1.5,-1);
	\draw[green,-,ultra thick] (0,-1) to (-0.44,-0.5);
	\draw[green,-,ultra thick] (-0.44,-0.5) to (-0.44,1);
	\draw[red,-,ultra thick] (-1,-1) to (-.5,-.5);
	\draw[red,-,ultra thick] (-.5,-.5) to (-.5,1);
	\draw[densely dotted] (-1.2,-1) -- (1.7,-1);
	\draw[densely dotted] (-1.2,0) -- (1.7,0);
	\draw[densely dotted] (-1.2,1) -- (1.7,1);
\end{scope}
		\node at (-1,1.2) {$\scriptstyle B$};
		\node at (0,1.2) {$\scriptstyle B'$};
		\node at (.5,1.2) {$\scriptstyle C$};
		\node at (1.5,1.2) {$\scriptstyle C'$};
\end{tikzpicture}
+
\begin{tikzpicture}[scale=0.75,baseline = -3]
\begin{scope}[yscale=-1,xscale=-1]
	\draw[green,-,ultra thick] (.95,1) to (.95,.5);
	\draw[green,-,ultra thick] (.5,0) to (.95,.5);
	\draw[green,-,ultra thick] (.5,0) to (.5,-1);
	\draw[red,-,ultra thick] (1,1) to (1,.5);
	\draw[red,-,ultra thick] (1.5,0) to (1,.5);
	\draw[red,-,ultra thick] (1.5,0) to (1.5,-1);
	\draw[blue,-,ultra thick] (0,-1) to (-0.44,-0.5);
	\draw[blue,-,ultra thick] (-0.44,-0.5) to (-0.44,1);
	\draw[yellow,-,ultra thick] (-1,-1) to (-.5,-.5);
	\draw[yellow,-,ultra thick] (-.5,-.5) to (-.5,1);
	\draw[densely dotted] (-1.2,-1) -- (1.7,-1);
	\draw[densely dotted] (-1.2,0) -- (1.7,0);
	\draw[densely dotted] (-1.2,1) -- (1.7,1);
\end{scope}
		\node at (1,1.2) {$\scriptstyle C'$};
		\node at (0,1.2) {$\scriptstyle C$};
		\node at (-.5,1.2) {$\scriptstyle B'$};
		\node at (-1.5,1.2) {$\scriptstyle B$};
\end{tikzpicture}.
\]

\item[($\bigwedge$')] Suppose that $\sigma \lessdot \pi$ and that $\pi$ is obtained from $\sigma$ by merging two blocks $B,B'$ into one.
Then:
\[ (\sigma,\pi,\sigma) = |B| |B'| (\sigma). \]
Diagrammatically, we can express this as:
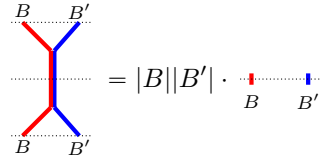
\[
\begin{tikzpicture}[scale=0.75,baseline = -3]
	\draw[blue,-,ultra thick] (0,-1) to (-0.44,-0.5);
	\draw[blue,-,ultra thick] (-0.44,0.5) to (0,1);
	\draw[blue,-,ultra thick] (-0.44,-0.5) to (-0.44,0.5);
	\draw[red,-,ultra thick] (-.5,-0.5) to (-.5,.5);
	\draw[red,-,ultra thick] (-1,-1) to (-.5,-.5);
	\draw[red,-,ultra thick] (-.5,0.5) to (-1,1);
	\draw[densely dotted] (-1.2,-1) -- (.2,-1);
	\draw[densely dotted] (-1.2,0) -- (.2,0);
	\draw[densely dotted] (-1.2,1) -- (.2,1);
	\node at (-1,-1.2) {$\scriptstyle B$};
        \node at (0,-1.2) {$\scriptstyle B'$};
	\node at (-1,1.2) {$\scriptstyle B$};
        \node at (0,1.2) {$\scriptstyle B'$};
\end{tikzpicture}
= |B||B'| \cdot \begin{tikzpicture}[scale=0.75,baseline = -3]
	\draw[blue,-,ultra thick] (0,-.1) to (0,.1);
	\draw[red,-,ultra thick] (-1,-.1) to (-1,.1);
	\draw[densely dotted] (-1.2,0) -- (.3,0);
	\node at (-1,-.4) {$\scriptstyle B$};
        \node at (0,-.4) {$\scriptstyle B'$};
\end{tikzpicture}
\]

\item[($\bigwedge$'')] Suppose that $\sigma \lessdot \pi$ and $\sigma' \lessdot \pi$ and that $\pi$ is obtained from $\sigma$ (resp. $\sigma'$) by merging two blocks $B,B'$ (resp. $C,C'$) into one and that the meet $\tau$ of $\sigma$ and $\sigma'$ satisfies $|\tau| = |\sigma|+1$.
Then
\[ (\sigma,\pi,\sigma') =  - (\sigma,\tau,\sigma').\]
Note that this condition holds if either (after possibly relabelling the blocks) $B \cup B' = C \cup C'$ and $B' \subsetneq C'$, or $B\cup B'$ and $C \cup C'$ are disjoint.

Diagrammatically:
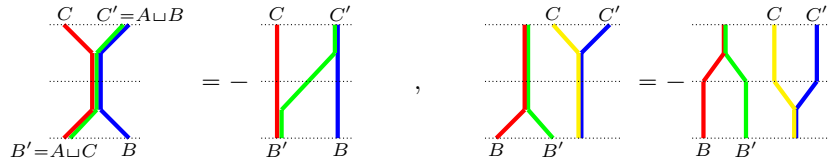
\[
\begin{tikzpicture}[scale=0.75,baseline = -3]
	\draw[blue,-,ultra thick] (.6,-1) to (0.1,-.5);
	\draw[blue,-,ultra thick] (.6,1) to (0.1,.5);
	\draw[green,-,ultra thick] (-0.44,-1) to (0.03,-.5);
	\draw[green,-,ultra thick] (.03,.5) to (0.5,1);
	\draw[red,-,ultra thick] (-.55,-1) to (-.05,-.5);
	\draw[red,-,ultra thick] (-.55,1) to (-.05,.5);
	\draw[green,-,ultra thick] (.03,-.5) to (0.03,.5);
	\draw[red,-,ultra thick] (-.05,-.5) to (-.05,.5);
	\draw[blue,-,ultra thick] (.1,-.5) to (0.1,.5);
	\draw[densely dotted] (-.8,-1) -- (.8,-1);
	\draw[densely dotted] (-.8,0) -- (.8,0);
	\draw[densely dotted] (-.8,1) -- (.8,1);
		\node at (-.725,-1.2) {$\scriptstyle B' = A \sqcup C$};
		\node at (.6,-1.2) {$\scriptstyle B$};
		\node at (-.525,1.2) {$\scriptstyle C$};
		\node at (.8,1.2) {$\scriptstyle C' = A \sqcup B$};
\end{tikzpicture}
= -~
\begin{tikzpicture}[scale=0.75,baseline = -3]
	\draw[blue,-,ultra thick] (.55,-1) to (.55,1);
	\draw[green,-,ultra thick] (-0.44,-.5) to (0.5,.5);
	\draw[green,-,ultra thick] (-0.44,-.5) to (-0.44,-1);
	\draw[green,-,ultra thick] (0.5,1) to (0.5,.5);
	\draw[red,-,ultra thick] (-.52,-1) to (-.52,1);
	\draw[densely dotted] (-.8,-1) -- (.8,-1);
	\draw[densely dotted] (-.8,0) -- (.8,0);
	\draw[densely dotted] (-.8,1) -- (.8,1);
		\node at (-.525,-1.2) {$\scriptstyle B'$};
		\node at (.6,-1.2) {$\scriptstyle B$};
		\node at (-.525,1.2) {$\scriptstyle C$};
		\node at (.6,1.2) {$\scriptstyle C'$};
\end{tikzpicture}\qquad ,\qquad 
\begin{tikzpicture}[scale=0.75,baseline = -3]
	\draw[blue,-,ultra thick] (.5,0.5) to (1,1);
	\draw[blue,-,ultra thick] (.5,0.5) to (.5,-1);
	\draw[yellow,-,ultra thick] (.45,.5) to (0,1);
	\draw[yellow,-,ultra thick] (.45,-1) to (.45,.5);
	\draw[green,-,ultra thick] (0,-1) to (-0.44,-.5);
	\draw[green,-,ultra thick] (-0.44,-.5) to (-0.44,1);
	\draw[red,-,ultra thick] (-1,-1) to (-.5,-.5);
	\draw[red,-,ultra thick] (-.5,-.5) to (-.5,1);
	\draw[densely dotted] (-1.2,-1) -- (1.2,-1);
	\draw[densely dotted] (-1.2,0) -- (1.2,0);
	\draw[densely dotted] (-1.2,1) -- (1.2,1);
		\node at (-1,-1.2) {$\scriptstyle B$};
		\node at (0,-1.2) {$\scriptstyle B'$};
		\node at (0,1.2) {$\scriptstyle C$};
		\node at (1,1.2) {$\scriptstyle C'$};
\end{tikzpicture} = -
\begin{tikzpicture}[scale=0.75,baseline = -3]
	\draw[blue,-,ultra thick] (1,0) to (1,1);
	\draw[blue,-,ultra thick] (1,0) to (.625,-.5);
	\draw[blue,-,ultra thick] (.635,-1) to (.635,-.5);
	\draw[yellow,-,ultra thick] (.25,0) to (0.25,1);
	\draw[yellow,-,ultra thick] (.62,-.5) to (.25,0);
	\draw[yellow,-,ultra thick] (.6,-.5) to 	(.6,-1);
	\draw[green,-,ultra thick] (-0.25,-1) to (-0.25,0);
	\draw[green,-,ultra thick] (-.25,0) to (-0.62,.5);
	\draw[red,-,ultra thick] (-1,-1) to (-1,0);
	\draw[red,-,ultra thick] (-1,0) to (-.635,.5);
	\draw[red,-,ultra thick] (-.64,1) to (-.64,.5);
	\draw[green,-,ultra thick] (-0.61,1) to (-0.61,.5);
	\draw[densely dotted] (-1.2,-1) -- (1.2,-1);
	\draw[densely dotted] (-1.2,0) -- (1.2,0);
	\draw[densely dotted] (-1.2,1) -- (1.2,1);
		\node at (-1,-1.2) {$\scriptstyle B$};
		\node at (-.25,-1.2) {$\scriptstyle B'$};
		\node at (.25,1.2) {$\scriptstyle C$};
		\node at (1,1.2) {$\scriptstyle C'$};
\end{tikzpicture}
\]

\item[($\bigwedge$''')] Finally if $\sigma, \sigma' \lessdot \pi$ and $\pi$ is obtained from $\sigma$ (resp. $\sigma'$) by merging two blocks $B,B'$ (resp. $C,C'$) into one, but $\sigma \neq \sigma'$ and $|\tau| \neq |\sigma|+1$ where $\tau$ is the meet of $\sigma$ and $\sigma'$, then:
\[ (\sigma,\pi,\sigma') = 0.
\]

Diagrammatically:
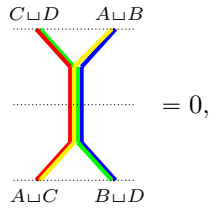
\[
\begin{tikzpicture}[baseline = -3]
	\draw[blue,-,ultra thick] (.55,-1) to (0.1,-.5);
	\draw[blue,-,ultra thick] (0.1,-.5) to (0.1,.5);
	\draw[blue,-,ultra thick] (.55,1) to (0.1,.5);
	\draw[green,-,ultra thick] (0.05,-.5) to (0.5,-1);
	\draw[green,-,ultra thick] (.05,-.5) to (0.05,.5);
	\draw[green,-,ultra thick] (0,.5) to (-0.45,1);
	\draw[red,-,ultra thick] (-.5,-1) to (-.05,-.5);
	\draw[red,-,ultra thick] (-.05,-.5) to (-.05,.5);
	\draw[red,-,ultra thick] (-.5,1) to (-.05,0.5);
	\draw[yellow,-,ultra thick] (-.45,-1) to (0,-.5);
	\draw[yellow,-,ultra thick] (0,-.5) to (0,.5);
	\draw[yellow,-,ultra thick] (0,.5) to (.5,1);
	\draw[densely dotted] (-.8,-1) -- (.8,-1);
	\draw[densely dotted] (-.8,0) -- (.8,0);
	\draw[densely dotted] (-.8,1) -- (.8,1);
		\node at (-.525,-1.2) {$\scriptstyle A \sqcup C$};
		\node at (.6,-1.2) {$\scriptstyle B \sqcup D$};
		\node at (-.525,1.2) {$\scriptstyle C \sqcup D$};
		\node at (.6,1.2) {$\scriptstyle A \sqcup B$};
\end{tikzpicture} = 0,
\]
for any nonempty disjoint sets $A,B,C,D$.
\end{itemize}
\end{definition}

For any $S(M(K_n))$-module $X$ and partition $\sigma$ of $[n]$, we can consider the \textit{weight space} $X^\sigma:= (\sigma)X \subseteq X$ and the corresponding weight decomposition $X = \bigoplus_{\sigma} X^\sigma$.

Our main results take the following form for the algebra $S(M(K_n))$:
 \begin{enumerate}
 \item For any field $k$, $S_k(M(K_n))=S(M(K_n))\otimes k$ is a quasihereditary algebra with highest weights indexed by the set partitions of $[n]$ (with the inverse refinement order so that the partition with a single block is minimal). 
 
 For each $\sigma \in \Pi_n$, let $\Fl^+_\sigma$ (resp. $\Fl^-_\sigma$) be the free $\Z$-module with basis the set of all flags $(\sigma_l \gtrdot \sigma_{l-1} \gtrdot \ldots \gtrdot \sigma_1 \gtrdot \sigma)$ (resp. $(\sigma \lessdot \sigma_{1} \lessdot \ldots \lessdot \sigma_{l-1} \gtrdot \sigma_l)$), where $l$ is allowed to vary.  Let $\Fs^{\pm}_\sigma$ be the quotient of $\Fl^{\pm}_\sigma$ by all flag space relations.  Let $(\Fs^+)_\sigma^\pi \subset \Fs^+_\sigma$ (resp. $(\Fs^-)_\sigma^\pi \subset \Fs^-_\sigma$) for the submodule generated by flags $(\pi \gtrdot \sigma_{l-1} \gtrdot \ldots \gtrdot \sigma_1 \gtrdot \sigma)$ (resp. in the reverse order).  In particular \[ \Fs^\pm_\sigma = \bigoplus_{\sigma \leq \pi} (\Fs^\pm)_\sigma^\pi.\]

In Section~\ref{sec-reedy} we prove that multiplication induces an isomorphism
\[ \bigoplus_{\sigma} \Fs^+_\sigma \otimes \Fs^-_\sigma \cong \bigoplus_{\pi' \geq \sigma \leq \pi} (\Fs^+)_\sigma^{\pi'} \otimes (\Fs^-)_\sigma^\pi \cong S(M(K_n)).\]

As for $S(PG(n-1,q))$, we construct a (graded) standard module $\Delta_\sigma$ for $S(PG(n-1,q))$ for each partition $\sigma \in \Pi_n$ such that 
$ \Delta_\sigma \cong \Fs_\sigma^+.$

\item The graded rank of $\Delta_\sigma$ is given by:
\[ \prod_{j=1}^{|\sigma|-1} (1+ j t).\]

The graded rank of $S(M)$ is therefore:
\[\sum_{i=1}^n S(n,i) \left(\prod_{j=1}^{i-1} (1+j t)(1+j t^{-1}) \right) ,\]
where $S(n,i)$ is the Stirling number of the second type, the number of set partitions of $[n]$ with $i$ blocks. 

\item As in the previous example there are bilinear forms \[\langle~,~\rangle_\sigma : \Delta_\sigma \times \Delta_\sigma \to \Z.\]  The standard module $\Delta_{\sigma,k}= \Delta_\sigma \otimes k$ has a simple quotient $L_{\sigma,k}:= \Delta_{\sigma,k}/\mathrm{rad}\langle~,~\rangle_{\sigma,k}$ and the set $\{L_{\sigma,k}\}_{\sigma \in \Pi_n}$ is a complete set of non-isomorphic simple $S_k(M(K_n))$-modules.  Additionally, one can again construct a natural Jantzen filtration on $\Delta_{\sigma,k}$.

\item Let $\ell$ be the characteristic of $k$.  For a set partition $\sigma= B_1 / B_2 / \ldots / B_l$, the standard module $\Delta_{\sigma,k}$ is not simple if and only if $0<\ell \leq |\sigma|=l$.   Consequently, $S_k(M(K_n))$ is not semisimple if and only if $0<\ell \leq n$.

\item Similarly to the action of $GL_n(\FM_q)$ on $\Delta_{\{0\}}^{\FM_q^n}$, the symmetric group $\mathfrak S_n$ acts on the weight space 
\[\Delta_{1/\ldots/n}^{[n]} \cong (\Fl^+)_{1/\ldots/n}^{[n]}, \]
which has rank $(n-1)!$.  The $\mathfrak S_n$-module $(\Fl^+)_{1/\ldots/n}^{[n]}$ is dual to the natural $\mathfrak S_n$-module structure on the Orlik-Solomon module or equivalently top homology of the partition lattice $\tilde{H}_{n-3}(\overline{\Pi}_n;k)$.  This later $\mathfrak S_n$-representation is known to be isomorphic to the representation $\mathrm{Lie}_k(n)$ of
$\mathfrak S_n$ on the multilinear component of the free Lie algebra over $k$ on $n$ generators tensored with the sign representation~\cite{WachsLie}.  By the same reasoning we saw in the $PG(n-1,q)$-case, the restriction of the Jantzen filtration to $(\Fl^+)_{1/\ldots/n}^{[n]}$ is compatible with the action of $\mathfrak S_n$ and nontrivial whenever $\ell\leq n$.

It would be interesting to see if the structure of standard $S_k(M(K_n))$-modules could be used to study the structure of the $\mathfrak{S}_n$-module $\mathrm{Lie}_k(n)$ or vice versa. 
\end{enumerate}

%%%%%%%%%%
\section{Matroid background}
\label{sec-matroids}
%%%%%%%%%%

Before describing our results for general matroids, we recall the necessary definitions from the theory of matroids.  Readers who are already familiar with matroids are encouraged to skip this section.

Let $M$ be a matroid with underlying finite set $E$ and lattice of flats $\cL$.  Let $\rho$ be the rank function on $\cL$.

For any subset $X \subseteq E$, the \textit{closure} $\overline X$ (or span) of $X$ is defined to be the minimal flat containing $X$.

If $X,Y \in \Lc$, then the \textit{join} $X \vee Y$ is $\overline{X \cup Y}$ and the \textit{meet} $X \wedge Y$ is the intersection $X \cap Y$.  The lattice of flats $\Lc$ is \textit{semimodular} meaning that for any $X,Y \in \Lc$,
\[ \rho(X) + \rho(Y) \geq \rho(X \vee Y) + \rho(X \wedge Y).\]
When equality holds we say that $X$ and $Y$ form a \textit{modular pair} of flats.  If $X, Y$ form a modular pair then, $$X \geq Z \quad \mathrm{implies} \quad X \cap (Y \vee Z) = (X \cap Y) \vee Z.$$

We extend $\rho$ to arbitrary subsets $X \subseteq E$ by letting the rank $\rho(X) = \rho(\overline{X})$.

A subset $I \subseteq E$ is said to be \textit{independent} if $\rho(I) = |I|$.  We denote by $\cI$ the collection of independent sets of $M$.

The maximal elements of $\cI$ (with respect to inclusion) are called \textit{bases} of $M$ and all bases have the same number of elements, namely $\rho(M):=\rho(E)$ the rank of $M$, which we will also denote by $r$. Let $\Bas(M)$ denote the set of all bases of $M$.

A set $X \subseteq E$ is \textit{dependent} if $X \not\in \cI$.  The minimal dependent sets of $M$ are called \textit{circuits}.

\begin{example*}
\begin{enumerate}
\item  For the projective geometry $PG(n-1,q)$, a subset $S\subseteq E$ is independent if the corresponding lines span a subspace of dimension $|S|$.  Note that the rank of $PG(n-1,q)$ is $n$.
\item For the graphic matroid $M(K_n)$ the independent sets are the subsets of edges of $K_n$ that contain no closed circuits. The bases of $M(K_n)$ correspond to spanning trees and the rank of $M(K_n)$ is $n-1$.
\end{enumerate}
\end{example*}

For any subset $X \subseteq E$, the collection $\cL^X = \{F \cap X ~|~ F \in \cL\}$ forms the set of flats for a matroid $M^X$ with underlying set $X$ called the \textit{restriction} of $M$ to $X$.  
The collection $\cL_X = \{F \setminus X~|~F \in \cL\}$ forms the set of flats for a matroid $M_X$ with underlying set $E\setminus X$ called the \textit{contraction} of $M$ by $X$.

An element $s \in E$ is called a \textit{loop} of $M$ if $s$ is contained in the minimal flat of $M$ (equivalently, the set $\{s\}$ is dependent).  An element $t \in E$ is called a \textit{coloop} if $E\setminus t$ is a flat of $M$ (equivalently, $s$ is contained in every basis of $M$). 

A flat $X \in \Lc$ (or more generally a set) is said to be \textit{cyclic} (or coloop-free) if no element of $X$ is a coloop of the restriction $M^X$.  Equivalently, a flat (or set) is cyclic if and only if it can be expressed as a union of circuits of $M$.  Let $\mathrm{Cyc}(M)$ denote the collection of cyclic sets and $\F= \F(M)$ denote the poset of all cyclic flats ordered by inclusion.  We note that $\F$ also forms a lattice, but where the meet of $X$ and $Y$ is the obtained from $X \cap Y$ by removing all coloops.

\begin{example*}(Continued)
\begin{enumerate}
\item In $PG(n-1,q)$ there are no loops and every flat of rank not equal to one is cyclic, but every rank one flat contains a single element, which is a coloop. 
\item In $M(K_n)$, the flat corresponding to a set partition $\sigma= B_1 / B_2 / \ldots / B_l$ is cyclic if and only if $\sigma$ has no blocks of 2 elements.
\end{enumerate}
\end{example*}

The \textit{dual matroid} $M^*$ of $M=(E,\cL)$ is defined by $M^* = (E,\cL^*)$, where $\cL^*= \{E \setminus C~|~C \in \mathrm{Cyc}(M)\}$.
The set $\Bas^*(M) = \{ E \setminus X ~|~ X \in \Bas(M)\}$ forms the set of bases for $M^*$.  Note that $M^*$ has rank $\rho(M^*) = |E|-r$.

If $X \subset E$, one can check that the contraction $M_X$ is equal to $((M^*)^{E\setminus X})^*.$

The \textit{direct sum} of matroids $M_1$ and $M_2$ on sets $E_1$ and $E_2$ is the matroid on $E_1 \coprod E_2$ whose flats are the unions of flats of $M_1$ and $M_2$.
A matroid is \textit{connected} if it cannot be written as a direct sum.

The \textit{M\"obius function} of the matroid is the function $\mu: \Lc \times \Lc \to \Z$ defined recursively by
\[ \mu(x,x) = 1 \textrm{~for~all~}x \in \Lc, \qquad \mu(x,y)=0 \textrm{~if~}x \not\leq y,\]
\[ \mu(x,y) = - \sum_{x \leq z <y} \mu(x,z)~\textrm{for}~x <y.\]
We abbreviate $\mu(x):= \mu(\varnothing, x)$ and write $\mu^+(x) = (-1)^{\rho(x)} \mu(x)$ for the unsigned M\"obius function (which is always positive).

The \textit{characteristic polynomial} of $M$ is defined in terms of the M\"obius function as \[p(M;t) = \sum_{F \in \Lc} \mu(F) t^{\rho(M)-\rho(F)}.\]

Finally, we mention Crapo's \textit{beta invariant}, which can be defined as $$\beta(M) = (-1)^{\rho(M)-1}\frac{d}{dt}p(M;1) = (-1)^{\rho(M)} \sum_{F \in \Lc} \mu(F) \rho(F).$$
It turns out that $\beta(M)$ is always non-negative and equal to zero if and only if $M$ is not connected or is a loop.  More detailed information on the characteristic polynomial, beta invariant and related topics can be found in the survey~\cite{zaslavsky-mu}.

%%%%%%%%%%%%%%%%%%%
\section{Reedy decomposition}
\label{sec-reedy}
%%%%%%%%%%%%%%%%%%%

To prove that the algebra is cellular and quasi-hereditary, it will be convenient to prove that it has a triangular decomposition of the form recently defined in~\cite{lin-reedy} under the name \textit{Reedy decomposition} and shown in~\cite{CDK} to be a special case of the triangular decompositions described for example in~\cite[Definition 5.31]{BS}.

\begin{definition}
Let $k$ be a field.  Let $A$ be a finite-dimensional $k$-algebra with a complete set $E$ of pairwise orthogonal idempotents.  $A$ is a \textit{Reedy algebra} if there is a degree function $\deg:E \to \mathbb N$ and two subalgebras $A^+$ and $A^-$ (with the same unit element as $A$) such that:
\begin{enumerate}
\item \label{reedy+} for all $e \in E$, there is an isomorphism $e A^+ e \cong k$ of $k$-vector spaces and for $e_i, e_j \in E$ such that $e_j \neq e_i$, if $e_j A^+ e_i \neq 0$, then $\deg(e_j) > \deg(e_i)$, 
\item \label{reedy-} for all $e \in E$, there is an isomorphism $e A^- e \cong k$ of $k$-vector spaces and for $e_i, e_j \in E$ such that $e_j \neq e_i$, if $e_j A^- e_i \neq 0$, then $\deg(e_j) < \deg(e_i)$. 
\item \label{reedy+-} for any $e_i,e_j \in E$, multiplication in $A$ induces an isomorphism of $k$-modules:
\[ \bigoplus_{e_l} e_j A^+ e_l \otimes e_l A^- e_i \stackrel{\cong}{\longrightarrow} e_j A e_i.\]
\end{enumerate}
\end{definition}

\begin{theorem}
For any field $k$, the algebra $S_k(M)$ is a Reedy algebra and so quasi-hereditary with a triangular decomposition.
\end{theorem}

\begin{proof}
While the above definition is given for a field, we may check it the analogous statement over $\Z$.

The set $\Lc$ of flats of $M$ gives a complete set of pairwise orthogonal idempotents.  Let $\deg: \Lc \to \mathbb N$ be the rank function. Let $\Fs^+$ denote the subalgebra of $S(M)$ generated by increasing flag paths $(K_k \gtrdot \ldots \gtrdot K_1 \gtrdot K_0)$ and let $\Fs^-$ denote the subalgebra generated by decreasing flag paths $(K_0 \lessdot \ldots \lessdot K_{k-1} \lessdot K_k)$).  Then properties (\ref{reedy+}) and (\ref{reedy-}) are immediate from the definition of $S(M)$.

It remains to show the property (\ref{reedy+-}), that multiplication induces an isomorphism of $\Z$-modules:
\[ \bigoplus_{L \subseteq K,J} (K) \Fs^+ (L) \otimes (L) \Fs^- (J) \stackrel{\cong}{\longrightarrow} (K) S(M) (J).\]

Consider first the quotient $\widetilde{S}(M)$ of $\FlP(M)$ by just the relation ($\Lambda$).  Let $\Fl^+$ (resp. $\Fl^-$) denote the subalgebra of $\widetilde{S}(M)$ generated by increasing (resp. decreasing) flags $F = (K_k \gtrdot \ldots \gtrdot K_1)$ (resp. $(K_k \lessdot  \ldots \lessdot K_1)$).  For each flag path $(K_k, \ldots, K_1)$, by repeatedly applying the relation ($\Lambda$) to any part of the path of the form $K_{i+1} \lessdot K_i \gtrdot K_{i-1}$, we will eventually either get $0$ or a non-zero scalar multiple of a flag path of the form $$(K_k \gtrdot \ldots \gtrdot J) \cdot (J \lessdot \ldots \lessdot K_1) \in (K_k) \Fl^+ (J) \Fl^- (K_1).$$  Thus the map induced by multiplication:
\[ \bigoplus_{L \subseteq K,J} (K) \Fl^+ (L) \otimes (L) \Fl^- (J) \longrightarrow (K) \widetilde{S}(M) (J)\]
is surjective.

Note that if there are multiple parts of the path of the form $K_{i+1} \lessdot K_i \gtrdot K_{i-1}$, applying the relation in either order gives the same result and so if the final result is non-zero, then the flag path $(K_k \gtrdot \ldots \gtrdot J) \cdot (J \lessdot \ldots \lessdot K_1)$ is uniquely determined and the set of such paths forms a basis for $\widetilde{S}(M)$.  Thus the map induced by multiplication is also injective and an isomorphism of $\Z$-modules.

Imposing the additional relation (FS) gives the projection $\widetilde{S}(M) \to S(M)$ and induced maps $\Fl^+ \to \Fs^+$ and $\Fl^- \to \Fs^-$.  Note that these maps complete the commutative diagram:

\[\begin{tikzcd}
 \bigoplus_{L \subseteq K,J} (K) \Fl^+ (L) \otimes (L) \Fl^- (J) \arrow[d,"",two heads]\arrow[r,"\cong"] & (K) \widetilde{S}(M) (J) \arrow[d,"",two heads] \\
\bigoplus_{L \subseteq K,J} (K) \Fs^+ (L) \otimes (L) \Fs^- (J) \arrow[r,""] & (K) {S}(M) (J).
\end{tikzcd}
\]

It remains to show that the bottom map is an injection.  To do so we need to show that imposing the relation (FS) on $\widetilde{S}(M)$ does not impose any additional relations on $\bigoplus_{L \subset K,J} (K) \Fl^+ (L) \otimes (L) \Fl^- (J)$ other than the relation (FS) on $\Fl^+$ and $\Fl^-$ individually.

Consider an arbitrary flag path with a positive or negative $i$-gap, meaning a sequence of flats
\[ F = (K_l , \ldots , K_{i+1} , K_{i-1} , \ldots , K_1), \]
where $(K_{i-1} , \ldots , K_1)$ and $(K_{k} , \ldots , K_{i+1})$ are flag paths and either $K_{i+1}>K_{i-1}$ with $\rho(K_{i+1}) = \rho(K_{i-1}) + 2$ or $K_{i+1}<K_{i-1}$ with $\rho(K_{i+1}) = \rho(K_{i-1}) - 2$.

By relation (FS), the sum 
$$\sum_{K_{i+1} \gtrdot K \gtrdot K_{i-1}} (K_l , \ldots , K_{i+1} ,K, K_{i-1} , \ldots , K_1)$$
is equal to zero in $S(M)$.  On the other hand, in $\widetilde S(M)$, each summand $$(K_l , \ldots , K_{i+1} ,K, K_{i+1} , \ldots , K_1)$$ is equal to a scalar multiple of a flag path of the form $$(K_l \gtrdot \ldots \gtrdot J) \cdot (J \lessdot \ldots \lessdot K_1) \in (K_l) \Fl^+ (J) \Fl^- (K_1).$$  We wish to show that the resulting sum (after applying ($\Lambda$) as needed) in $\bigoplus_{L \subset K_1,K_{l}} (K_l) \Fl^+ (L) \otimes (L) \Fl^- (K_1)$ maps to zero in $\bigoplus_{L \subset K_1,K_l} (K_l) \Fs^+ (L) \otimes (L) \Fs^- (K_1)$.  

Suppose from some $j$, $K_{j+1} \lessdot K_j \gtrdot K_{j-1}$. Unless $j=i+1$ or $i-1$, after applying the relation ($\Lambda$), the resulting sum is still (a multiple of) a sum over all ways of completing a flag path with a gap.  Thus we are reduced to studying the case $j=i+1$ (where $\rho(K_{i+1}) = \rho(K_{i-1}) + 2$) and $j=i-1$ (where $\rho(K_{i+1}) = \rho(K_{i-1}) - 2$).

As the relations (FS) and ($\bigwedge$) are symmetric (i.e., preserved when the flag paths are reversed), it is enough to consider the case $j=i+1$ with $\rho(K_{i+1}) = \rho(K_{i-1}) + 2$.  Moreover, as the relations are also local we can simplify to only considering the case of 4-step flags.  We write $X= K_{i+2}, Z = K_{i+1}, Y = K_{i-1}$ to simplify notation and consider:
$$\sum_{Z \gtrdot K \gtrdot Y} (X \lessdot Z \gtrdot K \gtrdot Y).$$

Then either:
\begin{enumerate}
\item \label{case-contain} $X \gtrdot Y$,
\item \label{case-notmodular} $(X, Y)$ is not a modular pair, or
\item \label{case-modular} $(X,Y)$ is a modular pair and $X \not\gtrdot Y$
\end{enumerate}

In case~(\ref{case-contain}) then $Z \gtrdot X \gtrdot Y$ and so:
\[ \sum_{Z \gtrdot K \gtrdot Y} (X \lessdot Z \gtrdot K \gtrdot Y) = (X \lessdot Z \gtrdot X \gtrdot Y) + \sum_{K \neq X} (X \lessdot Z \gtrdot K \gtrdot Y) \]

Applying ($\Lambda$) to the first term gives $|Z \setminus X| (X , Y)$.  Applying ($\bigwedge$) twice to the other terms gives: 
\[(X \lessdot Z \gtrdot K \gtrdot Y) = -(X \gtrdot Y \lessdot K \gtrdot Y) = -|K \setminus Y| ( X \gtrdot Y). \]

Thus the sum gives:
\[ = \left( |Z \setminus X|-\sum_{K \neq X}  |K \setminus Y| \right) ( X \gtrdot Y).\]
As $Z \setminus Y$ is partitioned by the sets $K \setminus Y$ where $Z \gtrdot K \gtrdot Y$, the coefficient is equal to zero and so the sum vanishes in $\FlP(M)$.

In case~(\ref{case-notmodular}) applying ($\bigwedge$) (twice if needed) also gives zero.

In case~(\ref{case-modular}): As $X \not\gtrdot Y$, $X \neq K$ for any $Z \gtrdot K \gtrdot Y$. Thus either $(X,K)$ is modular and $(X ,Z ,K) = - (X , X\cap K ,K)$ or  $(X,K)$ is not modular and  $(X , Z ,K) = 0$.  Our sum then becomes:
$$\sum_{Z \gtrdot K \gtrdot Y} (X \lessdot Z \gtrdot K \gtrdot Y) = \sum_{Z \gtrdot K \gtrdot Y, (X,K)~\mathrm{modular}} - (X \gtrdot X \cap K \lessdot K \gtrdot Y) $$
$$= \sum_{Z \gtrdot K \gtrdot Y, (X,K)~\mathrm{modular}} (X \gtrdot X \cap K \gtrdot X \cap Y \lessdot Y) $$
To finish, we need to show that the map: $$\{ K \in \Lc ~|~Z \gtrdot K \gtrdot Y,(X,K)~\mathrm{is~modular}\} \to \{J \in \Lc~|~ X \gtrdot J \gtrdot X \cap Y\}$$ defined by $K \mapsto X \cap K$ is a bijection with inverse $J \mapsto Y \vee J$.  If $X \gtrdot J \gtrdot X \cap Y$, then $$X \cap (Y \vee J) = (X \cap Y) \vee J = J,$$
where the first equality follow from the assumption that $(X,Y)$ is a modular pair and the second equality is immediate from $J \gtrdot X \cap Y$.  For the other composition, suppose that $(X,K)$ is a modular pair and $Z \gtrdot K \gtrdot Y$.  Then
\[ (K \cap X) \vee Y = K \cap (X \vee Y) = K \cap Z = K. \]

We conclude that 
$$\sum_{Z \gtrdot K \gtrdot Y} (X \lessdot Z \gtrdot K \gtrdot Y) = \sum_{X \gtrdot J \gtrdot X\cap Y} (X \gtrdot J\gtrdot X \cap Y \lessdot Y). $$

Therefore, after applying ($\bigwedge$) as many times as necessary to $$\sum_{K_{i+1} \gtrdot K \gtrdot K_{i-1}} (K_l , \ldots , K_{i+1} ,K, K_{i-1} , \ldots , K_1),$$ the resulting sum in $\bigoplus_{L} (K_l) \Fl^+ (L) \otimes (L) \Fl^- (K_1)$ will equal zero or reduce to zero under the map to $\bigoplus_{L} (K_l) \Fs^+ (L) \otimes (L) \Fs^- (K_1)$.
\end{proof}

%%%%%%%%%%%%%%%%%%%
\section{Multiplication and cellular structure}
\label{sec-mult}
%%%%%%%%%%%%%%%%%%%

By the result of the previous section, the set of flag paths of the form $(K \gtrdot \ldots \gtrdot J \lessdot \ldots \lessdot L) \in (K) \Fl^+ (J) \Fl^- (L)$ spans $S(M)$.  In this section we compute the multiplication formula in $\widetilde S(M)$ for such paths and use it to describe a cellular structure on $S(M)$.

For the next proposition we adopt the following notation.  Let $K_k \lessdot K_{k+1} \lessdot \ldots \lessdot K_n$ and $L_l \lessdot L_{l+1} \lessdot \ldots \lessdot L_n$ be flags of flats in $M$ such that $K_n=L_n$, $\rho(K_i)=i$, and $\rho(L_j)=j$.  Let $D_i := K_i \setminus K_{i-1}$ for $i = k+1, \ldots, n$, $E_l=L_l$, and $E_j := L_j \setminus L_{j-1}$ for $j =l+1, \ldots, n$.

\begin{proposition}\label{prop-mult}
In $\widetilde{S}(M)$ (and hence also in $S(M)$), the flag path
\[ (K_k \lessdot K_{k+1} \lessdot \ldots \lessdot K_n = L_n \gtrdot L_{n-1} \gtrdot \ldots \gtrdot L_l ) \]
is equal to zero if $(K_i,L_j)$ is not modular for some $i$ and $j$.  Otherwise, let 
$$\sigma_{K,L}(i):= \min\{j~|~D_i \cap E_j \neq \varnothing\}.$$
Then $\sigma_{K,L}$ defines a double coset $\sigma_{K,L} \in \mathfrak{S}_k \backslash \mathfrak{S}_n / \mathfrak{S}_l$.

If $(K_i,L_j)$ is modular for all $i$ and $j$ then
\[ (K_k \lessdot K_{k+1} \lessdot \ldots \lessdot K_n = L_n \gtrdot L_{n-1} \gtrdot \ldots \gtrdot L_l )  \] 
\[= \mathrm{sgn}(\sigma_{K,L}) \cdot \prod_{i} |D_i \cap E_{\sigma(i)} | \cdot F,\]
where $\mathrm{sgn}(\sigma_{K,L})$ is the sign of the minimal coset representative of $\sigma_{K,L}$, the product runs over all $i=k \ldots, n$ such that $\sigma(i)>l$, and $F$ is the flag path obtained from the sequence of flats $$(K_k , K_k \cap L_{n-1} , \ldots , K_k \cap L_{l+1} , K_k \cap L_l  , K_{k+1} \cap L_l , \ldots , K_{n-1} \cap L_l , L_l )$$
by removing any consecutive repetitions.
\end{proposition}

\begin{proof}
We first check the formula in the case $k=n-1$.  

Suppose that for some $j$, then pair $(K_{n-1},L_{n-j})$ is not modular.  We may take $j$ to be the minimal such $j$.  In particular, that means $K_{n-1} \not> L_{n-j}$ and so for any $i = 1, \ldots, j$, 
\[ K_{n-1} \geq K_{n-1} \cap L_{n-i+1} \neq L_{n-i} \geq L_{n-j}.\]
Repeatedly applying the down-up relation ($\bigwedge$) to $(K_{n-1} \lessdot L_n \gtrdot L_{n-1} \gtrdot \ldots \gtrdot L_l )$, we get:
\[(-1)^{j} (K_{n-1} \gtrdot K_{n-1} \cap L_{n-1} \gtrdot \ldots \gtrdot K_{n-1} \cap L_{n-j+1} \lessdot L_{n-j+1} \gtrdot L_{n-j} \gtrdot \ldots \gtrdot L_l ). \]
Having assumed that $(K_{n-1}, L_{n-j})$ is not modular, we have that $(K_{n-1}\cap L_{n-j-1} , L_{n-j})$ is not modular either and so the expression vanishes as expected.

\medskip

Now suppose that $(K_{n-1},L_{n-i})$ is modular for all $i$. 
By definition, $\sigma_{K,L}(n) = \min\{j~|~D_n \cap E_j \neq \varnothing\}$.  

If $\sigma_{K,L}(n) > l$, then for all $i$ such that $i > \sigma_{K,L}(n)$, we have that $K_{n-1}\cap L_{i}  \neq L_{i-1}$.  Let $j = \sigma_{K,L}(n)$.  Repeatedly applying the down-up relation $n-j$ times to $(K_{n-1} \lessdot L_n \gtrdot L_{n-1} \gtrdot \ldots \gtrdot L_l )$, we get:
\[(-1)^{n-j} (K_{n-1} \gtrdot K_{n-1} \cap L_{n-1} \gtrdot \ldots \gtrdot K_{n-1} \cap L_{j} \lessdot L_{j} \gtrdot L_{j-1} \gtrdot \ldots \gtrdot L_l ). \]
Our choice of $j$ implies that $L_{j-1} \subseteq K_{n-1} \cap L_j$.  By modularity the rank of $K_{n-1} \cap L_j$ is $j-1= \rho(L_{j-1})$, so $K_{n-1} \cap L_{j} = L_{j-1}$. Applying the down-up relation ($\bigwedge$) one more time to the previous expression we get
\[(-1)^{n-j} |L_{j} \setminus L_{j-1}| (K_{n-1} \gtrdot K_{n-1} \cap L_{n-1} \gtrdot \ldots \gtrdot K_{n-1} \cap L_{j} = L_{j-1} \gtrdot \ldots \gtrdot L_l ). \]
Note that $|L_{j} \setminus L_{j-1}| = |D_n \cap E_{j}|$ and $n-j$ is the number of inversions of the minimal coset representative of $\sigma_{K,L}$.  Thus the formula holds in this case as well.

\medskip

If $\sigma_{K,L}(n) = l$, then $n-l$ is the number of inversions of the minimal coset representative of $\sigma_{K,L}$ and repeatedly applying the down-up relation ($\bigwedge$) to $(K_{n-1} \lessdot L_n \gtrdot L_{n-1} \gtrdot \ldots \gtrdot L_l )$, we get:
\[(-1)^{n-l} (K_{n-1} \gtrdot K_{n-1} \cap L_{n-1} \gtrdot \ldots \gtrdot K_{n-1} \cap L_{l} \lessdot L_{l}). \]

This completes the case of $k=n-1$.  

\medskip

For the general case we then repeat the argument using an $\{k,k+1,\ldots, n\}\times \{l,l+1,\ldots, n\}$-sized grid with the first index decreasing from north to south and the second decreasing from west to east.  At the $(i,j)$-grid point we write $K_i \cap L_j$.  In particular $K_n=L_n$ is at the upper left corner, with the flag of $K$'s decreasing down the left hand side, and the flag of $L$'s along the top.   

Applying the up-down relation as many times as possible corresponds to converting the path that runs up the left-hand side and along the top to the one that goes across (to the right) along the bottom and then up along the right-hand side.  To keep track of the result, in each box of the grid we enter an integer.  Here we label each box of the grid by the coordinates of its upper left corner.  We enter
\begin{itemize}
\item $0$ if the lower-left and upper-right corners of the box do not form a modular pair, otherwise: 
\item $-1$ if the lower-left and upper-right corners of the box are not equal and neither is equal to the upper left corner, 
\item $|(K_i \cap L_{j}) \setminus (K_i \cap L_{j-1})| = |D_i \cap E_{j}|$ if the lower-left and upper-right corners are equal (but not the same as the upper-left corner), 
\item $1$ if the upper-left corner of the box is equal to the lower-left or upper-right corner of the box. 
\end{itemize}
The result is the product of the contents of all the boxes times the path along the bottom and then up along the right-hand side with any redundancies removed.  By the argument above, if the pair of flags is modular, then the $i$th row of boxes contributes a factor $|D_i \cap E_{\sigma_{K,L}(i)}|$ times $-1$ to the number of boxes to the left of $(\sigma_{K,L}(i),i)$ that don't lie below a box $(\sigma_{K,L}(j),j)$ (i.e., the number of $1\leq j<i$ such that $\sigma(j)>\sigma_{K,L}(i)$, where $\sigma$ is the minimal coset representative of $\sigma_{K,L}$).  The number of $-1$'s that occur is then the length of the minimal coset representative.
\end{proof}

We record the following special case used in Section~\ref{sec-projgeom}.

\begin{corollary}\label{cor-flagmult}
In $S(PG(n-1,q))$, for any pair of flags $V_\bullet = (W=V_m \gtrdot \ldots \gtrdot V_1 \gtrdot V),~~W_\bullet = (W=W_m \gtrdot \ldots \gtrdot W_1 \gtrdot V)$, the expression
$$\sigma_{V_\bullet,W_\bullet}(i):= \min\{j~|~D_i \cap E_j \neq \varnothing\}$$
 defines an element $\sigma = \sigma_{V_\bullet,W_\bullet} \in \mathfrak{S}_m$ and 
\[ (V \lessdot V_1 \lessdot \ldots \lessdot V_m=W = W_m \gtrdot W_{m-1} \gtrdot \ldots \gtrdot W_1 \gtrdot V )  \] 
\[= \mathrm{sgn}(\sigma) q^{m \dim V +\Bell(w_0 \sigma)} (V),\]
where $w_0 \in \mathfrak S_m$ denotes the longest element (i.e., $w_0(1)=m, w_0(2) = m-1, \ldots, w_0(m)=1$) and $\Bell$ is the length function on $\mathfrak S_m$.
\end{corollary}

\begin{proof}
The main thing that needs to be checked is that 
\[\prod_{i=1}^m |D_i \cap E_{\sigma(i)} | = q^{m \dim V+\Bell(w_0 \sigma) }. \]

Fix $1\leq i \leq m$ and let $j:= \sigma(i)$.  Then $D_i \cap E_{\sigma(i)} = (V_i \cap W_j) \setminus (V_i \cap W_{j-1})$, where $V_i \cap W_{j-1} = V_{i-1} \cap W_j$ is a hyperplane in $V_i \cap W_j$, and so $|D_i \cap E_{\sigma(i)}| = q^{\dim(V_i \cap W_j)-1}$.

Note that the codimension of $V_i \cap W_{\sigma(i)}$ in $V_i$ is equal to the number of $1\leq j <i$ such that $\sigma(j)>\sigma(i)$.  Thus the dimension of $V_i \cap W_{\sigma(i)}$ is equal to $\dim V_i = \dim V + i$ minus the number of $1\leq j <i$ such that $\sigma(j)>\sigma(i)$, or equivalently $\dim V+1$ plus the number of $1\leq j<i$ such that $\sigma(j)<\sigma(i)$.   Thus $|D_i \cap E_{\sigma(i)}|$ is equal to $q$ to the power $\dim V$ plus the number of $1\leq j<i$ such that $\sigma(j)<\sigma(i)$.

Taking the product over all $1\leq i \leq m$ gives $q$ to the power $m\dim V$ plus the number of pairs $(j,i)$ such that $j<i$ and $\sigma(j) < \sigma(i)$.  The number of such pairs is the number of non-inversions, or $\binom{m}{2}-\Bell(\sigma) = \Bell(w_0 \sigma)$.   Thus $\prod_{i=1}^m |D_i \cap E_{\sigma(i)} | = q^{\Bell(w_0 \sigma) + m \dim V}$.
\end{proof}

Using Proposition~\ref{prop-mult} we can show that $S(M)$ is a cellular algebra.  Recall:

\begin{definition}
Let $R$ be a commutative Noetherian integral domain. An $R$-algebra $A$ is called a \textit{cellular algebra} with cell datum $(\Lambda, M, C, i)$ where:
\begin{itemize}
\item $\Lambda$ is a finite partially ordered set,
\item for each $\lambda \in \Lambda$ there is an associated non-empty finite set $M(\lambda)$,
\item $i: A \to A$ is an $R$-linear anti-automorphism such that $i^2 = \mathrm{id}_A$, and
\item for each $\lambda \in \Lambda$ and pair of elements $S,T \in M(\lambda)$, there is an element $C^\lambda_{ST} \in A$,
\end{itemize}
such that the following conditions are satisfied:
\begin{enumerate}
\item \label{cell1} The set $\{C^\lambda_{ST} ~|~\lambda \in \Lambda,~S,T \in M(\lambda) \}$ is an $R$-basis for $A$.
\item \label{cell2} The involution $i$ sends $C^\lambda_{ST}$ to $C^\lambda_{TS}$.
\item \label{cell3} For each $\lambda \in \Lambda$, $S,T \in M(\lambda)$ and $a \in A$, the product $a C^\lambda_{ST}$ can be expressed as
\[ a C^\lambda_{ST} = \left( \sum_{U \in M(\lambda)} r_a(U, S) C^\lambda_{UT} \right) + r',\] 
where $r'$ is a linear combination of basis elements with upper index $\mu$ strictly smaller than $\lambda$, and where the coeﬃcients $r_a(U, S) \in R$ do not depend on $T$.
\end{enumerate}
\end{definition}

To describe a cell datum for $S(M)$, we fix a labelling of flats, meaning simply a map $\ell: \Lc \to E$ such that $\ell(K) \in K$ for each $K$ (for example, we could order the set $E$ and let $\ell(K)$ be the least element of $K$).\footnote{While this notation for the labelling clashes with our notation for the characteristic of $k$, we only use it in this section.}  

A flag $(K_k \gtrdot \ldots \gtrdot K_1 \gtrdot K_0)$ is \textit{neat} (with respect to $\ell$) if $\ell(K_i) \in K_i \setminus K_{i-1}$ for all $i=1, \ldots, k$.  Brylawski--Varchenko show~\cite[(2.12)-(2.13)]{BV} the set of all neat flags (for a fixed $\ell$) is a $\Z$-basis for $\Fs^+$.  

Let $\iota$ denote the involution of $S(M)$ (and $\widetilde S(M)$) that sends a flag path $F=(K_1, \ldots, K_k)$ to the flag path $\iota(F) = (K_k, \ldots,K_1)$ obtained by listing the same flats in the reverse order.  

For each flat $K$, let $N(K)$ denote the set of neat flags $$(K_k \gtrdot K_{k-1} \gtrdot \ldots \gtrdot K_1 \gtrdot K).$$
For any pair of elements $F = (K_k , \ldots , K_1 , K), G = (L_l , \ldots , L_1, K) \in N(K)$, let 
\[C^K_{FG} = F \cdot \iota(G) =  (K_k \gtrdot \ldots \gtrdot K_1 \gtrdot K \lessdot L_1 \lessdot \ldots \lessdot L_l)\]

\begin{theorem}\label{thm-cell}
$(\Lc, N, C, \iota)$ is a cell datum for $S(M)$.
\end{theorem}

\begin{proof} We check conditions (\ref{cell1})-(\ref{cell3}). Brylawski--Varchenko's basis result~\cite[(2.12)-(2.13)]{BV} together with the decomposition from Section~\ref{sec-reedy} implies that the set $$\coprod_{K \in \Lc} \{C^K_{FG}~|~F,G \in N(K)\}$$ is a $\Z$-basis for $S(M)$.

By definition, $\iota(C^K_{FG}) = \iota(F \cdot \iota(G)) = \iota^2(G) \cdot \iota(F) = G \cdot \iota(F) = C^K_{GF}$.

To check (\ref{cell3}), fix a basis element $C^K_{FG}$.  It suffices to consider the case $a = F' \cdot \iota G'$, for flags 
$$F' = (K'_k \gtrdot \ldots \gtrdot K'_1 \gtrdot K'), \quad G' = (L'_l \gtrdot \ldots \gtrdot L'_1\gtrdot K').$$
Applying Proposition~\ref{prop-mult} to $\iota G' \cdot F$ in the product
\[ a C^K_{FG} = F' \cdot (\iota G' \cdot F) \cdot \iota G,\]
the result is either zero or a scalar multiple of $F' \cdot F'' \cdot \iota G'' \cdot \iota G$ for some 
$$F'' = (K' \gtrdot K''_{k'} \gtrdot \ldots \gtrdot K''_1 \gtrdot K''), \quad G'' = (K \gtrdot L''_{l'} \gtrdot \ldots \gtrdot L''_1\gtrdot K'').$$
Applying the relation $(FS)$ as need, $a C^K_{FG} $ is a linear combination of basis elements with upper index $K'' \leq K$.  Note that $K''=K$ only if $G'' = (K)$ and so in that case $a C^K_{FG} =  F' \cdot F'' \cdot \iota G$ is a linear combination $\sum_{H} r_a(H, F) C^K_{HG}$ and the coefficients $r_a(H, F) \in \Z$ do not depend on $G$. \end{proof}

%%%%%%%%%%%%%%%%%%%
\section{Standard modules}
\label{sec-modules}
%%%%%%%%%%%%%%%%%%%

For each flat $K \in \Lc$, the standard (or cell) module $\Delta_K$ is defined as $\Fs^+\cdot (K)$ viewed as the quotient of $S(M)\cdot (K)$ by the image of all flag paths that pass through a flat that does not contain $K$.  Equivalently, the standard module $\Delta_K$ can be viewed as the free $\Z$-module with basis $N(K)$, the set of neat flags \[(K_k \gtrdot K_{k-1} \gtrdot \ldots \gtrdot K_1 \gtrdot K),\] with the $S(M)$-module structure defined by:
\[ a \cdot F = \sum_{H \in N(K)} r_a(H, F) H, \]
for the coefficients $r_a(H,F) \in \Z$ defined by the cell datum.  In particular, for the minimal flat $K_0$ of $M$, the $S(M)$-module structure of $\Delta_{K_0}$ comes from the fact that $\Delta_{K_0} = \Fs^+ \cdot (K_0) \subset S(M)$ is a left ideal.

As we discussed in the examples, each $S(M)$-module has a natural weight space decomposition.   For any $S(M)$-module $X$ and flat $K \in \Lc(M)$, we can consider the weight space $X^K:= (K)X \subseteq X$ and the corresponding weight space decomposition $X = \bigoplus_{K\in \Lc(M)} X^K$.

For example, $\Delta_{K} = \bigoplus_{H\supset K} \Delta_K^H$, where $\Delta_K^H =(H)\cdot \Fs^+ \cdot (K)$ is homogeneous of degree $\rho(H)-\rho(K)$.  By~\cite[Proposition 2.18]{BV}, the set of $(H) \cdot \Fs^+ \cdot (K)$ is a free $\Z$-module of rank $\mu^+(M^H_K)$.  Thus the graded rank of $\Delta_K$ is
\[ \sum_{H \supset K} \mu^+(K,H) t^{\rho(H)-\rho(K)} = (-t)^{\rho(M_K)} p(M_K;-t^{-1}).\]

By the Reedy decomposition from Section~\ref{sec-reedy},
$$S(M) \cong \bigoplus_{L<K,J} (K) \Fs^+ (L) \otimes (L) \Fs^- (J) \cong \bigoplus_{L} \Delta_L \otimes \iota( \Delta_L).$$
We conclude that $S(M)$ is a free $\Z$-module of graded rank $\sum_F p(M_F;-t) p(M_F;-t^{-1}).$

Each cell module $\Delta_K$ inherits a natural symmetric bilinear form $$\langle~,~\rangle_K: \Delta_K \times \Delta_K \to \Z,$$ defined on flags $F=(K_k \gtrdot \ldots \gtrdot K_1 \gtrdot K), G=(L_l \gtrdot \ldots \gtrdot L_1 \gtrdot K) \in \Delta_K$ by:
\[ \langle F,G \rangle_K (K) =  \iota F \cdot G = (K \lessdot K_1 \lessdot \ldots \lessdot K_k) \cdot (L_l \gtrdot \ldots \gtrdot L_1 \gtrdot K).\]
In particular, the pairing vanishes when $K_k \neq L_l$ (and so $\Delta_{K} = \bigoplus_{H \supset K} (H) \Fs^+ (K)$ is an orthogonal decomposition).  By Proposition~\ref{prop-mult} and~\cite[Proposition 3.8]{BV} the bilinear form $\langle F,G \rangle_K$ is equal to Brylawski--Varchenko's bilinear form  $B(F,G)$ on flag space for the contraction $M_K$.

Now let $k$ be a field.  We may also define the standard module $\Delta_{K,k} = \Delta_K \otimes_\Z k$ for $S_k(M)$.

Let $\mathrm{rad}(K) = \{x \in \Delta_K ~|~ \langle x,y \rangle_K = 0~\mathrm{for~all}~y \in \Delta_K\}$ and $L_{K,k} := \Delta_{K,k}/\mathrm{rad}(K)$.  Note that $\langle (K),(K) \rangle_K = 1 \neq 0$, so $L_{K,k} \neq 0$. 

By the theory of cellular algebras, the set $\{L_{K,k}~|~ K \in \Lc\}$ is a complete set of non-isomorphic absolutely irreducible $S_k(M)$-modules and each standard module $\Delta_{K,k}$ is indecomposable with simple cosocle $L_{K,k}$.

In particular, $\Delta_{K,k}$ is equal to the simple module $L_{K,k}$ if and only if the radical of the form $\langle ~,~ \rangle_K$ is trivial, or equivalently $\langle ~,~ \rangle_K$ is non-degenerate.  Comparing Proposition~\ref{prop-mult} and~\cite[Proposition 3.8]{BV}, we find that $\langle ~,~ \rangle_K$ is equal to Brylawski--Varchenko's bilinear form  $B$ on flag space for the contraction $M/K$.  By Brylawski--Varchenko's determinant formula \cite[Theorem 4.16]{BV}, the determinant of $B$ restricted to $(H) \cdot \Fs^+ \cdot (K)$ is equal to 
\[ \prod_{J \in \Lc, K < J \leq H} |J \setminus K|^{\beta(M^J_K)\mu^+(M^H_J)}. \]
As $\mu^+(M^H_J)$ is always positive and $\beta(M^J_K) \neq 0$ if and only if $M^J_K$ is connected, $\langle ~,~ \rangle_K$ restricted to $(H) \cdot \Fs^+ \cdot (K)$ is non-degenerate if and only if the characteristic of $k$ does not divide $|J \setminus K|$ for any flat $J$ such that $K < J \leq H$ and $M^J_K$ is connected.  We conclude:

\begin{theorem}\label{thm-det}
The standard module $\Delta_{K,k}$ is simple if and only if the characteristic of $k$ does not divide $|J \setminus K|$ for any flat $J>K$ such that $M^J_K$ is connected.  The algebra $S_k(M)$ is semisimple if and only if the characteristic of $k$ does not divide $|J \setminus K|$ for any pair of flats $J>K$ such that $M^J_K$ is connected.
\end{theorem}

As we did in the case of $S(PG(n-1,q))$, we also observe that the integral bilinear form $\Delta_K \times \Delta_K \to \Z$ induces a natural $S_k(M)$-module filtration (Jantzen filtration) on $\Delta_{K,k}$ when $k$ is a field $k$ of characteristic $\ell>0$.  
This `Jantzen' filtration on $\Delta_{K,k}$ is defined by
\[\Delta_{K,k} = (\Delta_{K,k})^0 \supset (\Delta_{K,k})^1 \supset \ldots \supset (\Delta_{K,k})^N=0,\]
where 
\[\Delta_K^i = \{ x \in \Delta_K~|~\langle x, v \rangle \in \ell^i \Z~\mathrm{for~all}~v\in \Delta_K \},\]
and $(\Delta_{K,k})^i$ be the image of $\Delta_K^i \otimes_\Z k$ in $\Delta_{K,k} = \Delta_K \otimes_\Z k$.

%%%%%%%%%%%%%%%%%%%
\section{Basis modules}
\label{sec-tilting}
%%%%%%%%%%%%%%%%%%%

Let $\Lambda(M) = \oplus_{p} \Lambda^p(M)$ be the exterior algebra over $\Z$ generated by the ground set $E$.  Fix a total order on $E$ and for a subset $S = \{s_1, s_2, \ldots, s_p\} \subset E$ ordered so that $s_i<s_{i+1}$, let $e_S$ denote the monomial $s_1 \wedge s_2 \wedge \ldots \wedge s_p \in \Lambda^p(M)$.  We endow $\Lambda(M)$ with the symmetric bilinear form $\langle~,~\rangle$ for which the monomials $e_S$ are orthonormal.

Let $\Lambda^p_q(M) \subset \Lambda^p(M)$ denote the submodule generated by monomials $e_S$ for all $S \subset E$ of cardinality $p$ and rank $q$.

Let $\B(M) = \Lambda^r_r(M)$ be the submodule generated by the monomials corresponding to bases of $M$.  For any $F, G \in \Lc$ let\footnote{In~\cite{BMmatroid}, we ordered the set of flats by reverse inclusion for geometric reasons and defined $\B^F_G:= \B(M^G_F)$.  Here I have chosen to abandon that convention as it seems unnecessarily confusing.}
\[ \B^G_F = \begin{cases} \B(M^G_F) & \mathrm{if~} F \subset G, \\
0 & \mathrm{otherwise.}
\end{cases}\]

For example, note that $\B^F_F = \B(M^F_F) = \Z\{e_{\varnothing}\}$, with generator corresponding to the empty basis.  If $F \lessdot G$, then the bases of $M^G_F$ are the elements of $G \setminus F$ and so $\B^G_F = \Z\cdot \{e_{\{i\}}~|~i\in G\setminus F)$.

We now fix a flat $F$ and let $\Bt_F = \bigoplus_{G \in \Lc} \B_F^G$.  Note that $\Bt_F$ has a basis labelled by the independent sets of $M_F$.  We will endow $\Bt_F$ with the structure of a $S(M)$-module.  To do so, we define the action for the generators of $S(M)$ and then check that the relations hold.

For $K \in \Lc$, let $\phi(K) \in \End(\Bt_F)$ be defined by: for $b \in \B_F^G$, let $\phi(K) (b) = b$ if $G=K$ and 0 otherwise.  

For $L \gtrdot K$, let $d_L^K = \sum_{i\in L\setminus K} i$.  We define the action of the operator $\phi(L,K)\in \End(\Bt_F)$ on $b\in \B^G_F$ by the equation: 
\[ \phi(L \gtrdot K) (b) = \begin{cases}
d^L_K \wedge b  \in \B_F^L & \mathrm{if}~K=G \\
0 & \mathrm{otherwise}.
\end{cases}
\]
Note that for any basis $B$ for $M^K_F$ and $i \in L \setminus K$, the union $B \cup \{i\}$ is a basis for $M^L_F$ and so $d^L_K \wedge b$ is a well-defined element of $\B_F^L$.

For $L \gtrdot K$ as above, let $\phi(K \lessdot L)\in \End(\Bt_F)$ be defined as the adjoint of the operator $\phi(L \gtrdot K)$ with respect to the pairing $\langle~,~\rangle$ on $\Bt_F$.  Explicitly, this means for a basis $B$ for $M^G_F$, $e_B \in \B_F^G$, and
\[ \phi(K \lessdot L) (e_B) = \begin{cases}
\varepsilon(i,B) e_{B\cap K} \in \B_F^K & \mathrm{if}~G=L, K \geq F,~\mathrm{and}~ B\cap K \in \Bas(M^K_F),\\
0 & \mathrm{otherwise},
\end{cases}
\]
where $\varepsilon(i,B)$ is the sign $-1$ to the number of $j \in B$ such that $j <i$.  In particular, if $B'$ is a basis for $M^K_F$ and $i \in L \setminus K$, then $\phi(K \lessdot L) (i \wedge e_{B'}) = e_{B'}$.  Note that as $B \cap K \subset K$ and linearly independent, the condition $B\cap K \in \Bas(M^K_F)$ is equivalent to $|B \setminus K|=1$.

For a general flag path $(K_1, \ldots, K_k)$, let $$\phi(K_1, \ldots, K_k) := \phi(K_1,K_2)\circ \phi(K_2,K_3) \circ \ldots \circ \phi(K_{k-1},K_k).$$

\begin{proposition}
The actions described above satisfy the relations (FS) and ($\bigwedge$) relations and thus define a $S(M)$-module structure on $\Bt_F$.
\end{proposition}

\begin{proof} We begin with the relation (FS).  Suppose that $J,L \in \Lc$, $J < L$ and $\rho(L) = \rho(J)+2$.  We wish to show that $\sum_K \phi(L,K,J)$ acts on $\Bt_E$ by zero, where the sum is over all flats $K$ such that $L \gtrdot K \gtrdot J$. Note that such $K$ are naturally in bijection with the rank 1 flats of $M^L_J$.

The operator $\phi(L,K,J) = \phi(L,K) \circ \phi(K,J) \in \End(\Bt_F)$ acts on $\B^F_H$ by 0 unless $H=J$, in which case it acts as the wedge product on the left by 
$$\left(\sum_{i\in L\setminus K} i \right)\wedge \left(\sum_{j\in K\setminus J} j \right)=\sum_{i\in  L\setminus K, j\in K\setminus J} i \wedge j.$$

Now consider the sum over all intermediate flats $L \gtrdot K \gtrdot J$,
$$\sum_{L \gtrdot K \gtrdot J} \left(\sum_{i\in L\setminus K, j\in K\setminus J} i \wedge j \right).$$

As $M^L_J$ is a rank 2 matroid without loops, the set $L \setminus J$ is partitioned by the rank 1 flats of $M^L_J$ (the classes of parallel elements).  Thus a monomial $i\wedge j$ only appears in the sum if $\overline{\{i\}} \neq \overline{\{j\}}$.  In that case the monomial $i\wedge j$ appears only with positive multiplicity 1 in the summand labelled by $K = \overline{\{j\}}$ and with multiplicity $-1$  in the summand labelled by $K = \overline{\{i\}}$.  Thus the sum vanishes as claimed.  The case where $L < J$ and $\rho(L) = \rho(J)-2$ then follows by taking adjoints.  This completes the verification of the relation (FS).

\medskip

We now turn to the relation ($\bigwedge$).  Let $F=(J \lessdot K \gtrdot L)$, so in particular, $\rho(J) = \rho(L) = \rho(K)-1$.  We may assume that $J,L \geq F$ as otherwise both sides of the relation vanish.

If $J=L$, then by our definition, the operator $\phi(J, K, J) = \phi(J, K)\phi(K, J)$ acts on $\B_F^H$ by 0 unless $H=J$. For $e_B \in \B_F^J$ it acts by:
\[ \phi(J \lessdot K \gtrdot J)(e_B) =  \sum_{i\in K\setminus J} \phi(J \lessdot K)( i \wedge e_B) = \sum_{i\in K\setminus J} e_B = |K \setminus J| e_B. \]
Thus $\phi(J \lessdot K \gtrdot J) =  |K\setminus J| \cdot \phi(J)$ and so the relation holds in this case.

Now consider the case $J\neq L$.
By definition, the operator $\phi(J \lessdot K \gtrdot L)$ acts on $\B^H_F$ by 0 if $H \neq L$.  If $e_B \in \B^L_F$, where $B$ is a basis for $M^L_F$, then
\[ \phi(J \lessdot K \gtrdot L)(e_B)  =  \sum_{j\in K\setminus L} \phi(J \lessdot K) \left(j \wedge e_B \right).\]

The term $\phi(J \lessdot K) \left(j \wedge e_B \right)$ is equal to 0 unless $(\{j\} \cup B)\cap J$ is a basis for $M^J_F$.  In that case, as $K\neq L$, the set $(B \cup \{j\})\cap J$ has the same cardinality as, but is not equal to, $B$.  Thus if $\phi(J \lessdot K) \left(j \wedge e_B \right)\neq 0$, then $j \in J$ and there exists a unique $i \in B \setminus J$.  Thus $B' := B \setminus \{i\} = J \cap B$ is a basis for $M^{J \cap L}_F$ of cardinality $\rho(M^J_F)-1 = \rho(M^L_F)-1$, which implies that $J$ and $L$ form a modular pair.  We conclude that if $J$ and $L$ do not form a modular pair, then all of the terms $\phi(J \lessdot K) \left(j \wedge e_B \right)$ vanish and so $\phi(J \lessdot K \gtrdot L)=0$.

Now assume that $J$ and $L$ form a modular pair.  Suppose $B$ is a basis for $\B^L_F$.  If $|B \setminus J|>1$, then both $\phi(J \lessdot K \gtrdot L)(e_B)$ and $-\phi(J\gtrdot J\cap L \lessdot L)(e_B)$ vanish.  If $|B \setminus J|=1$, let $i\in B \setminus J$ be the unique element and $B' = B \cap J = B \setminus \{i\}$.  Then $B'$ is a basis for $M^{J \cap L}_F$ and we compute:
\[ \phi(J \lessdot K \gtrdot L)(i \wedge e_{B'}) = \phi(J\lessdot K) \left( \sum_{j \in K \setminus L} j \wedge i \wedge e_{B'} \right) =  \sum_{j \in K \setminus L} \phi(J\lessdot K)(j \wedge i \wedge e_{B'}) . \]
Recall that the term $\phi(J\lessdot K)(j \wedge i \wedge e_{B'})$ vanishes unless $j \in J$, so the sum is equal to
\[ \sum_{j \in J \setminus L} \phi(J\lessdot K)(j \wedge i \wedge e_{B'})  = - \sum_{j \in J \setminus L} \phi(J\lessdot K)(i \wedge j \wedge e_{B'}) = - \sum_{j \in J \setminus L}  j \wedge e_{B'}. \]
Similarly, applying $-\phi(J\gtrdot J\cap L \lessdot L)$ to $i \wedge e_{B'}$, we get the same result:
\[ -\phi(J\gtrdot J\cap L \lessdot L)(i \wedge e_{B'}) = -\phi(J\gtrdot J\cap L)(e_{B'}) = - \sum_{j \in J \setminus L}  j \wedge e_{B'}. \]
We conclude that $\phi(J \lessdot K \gtrdot L)$ and $-\phi(J\gtrdot J\cap L \lessdot L)$ agree on the set of all $e_B \in \B^L_F$ and thus are equal.  This completes the proof of relation ($\bigwedge$) and hence the homomorphism is well-defined. 
\end{proof}

\begin{remark}
I expect to show in a future paper that the $S_k(M)$-module $\Bt_F$ is tilting in the sense that it is self-dual and has a standard filtration and describe the Ringel dual of $S(M)$.
\end{remark}

%%%%%%%%%%%%%%%%%%%
\section{Matroidal Schur algebra}
\label{sec-schur}
%%%%%%%%%%%%%%%%%%%

In this section we recall the definition of the matroidal Schur algebra $\Rc(M)$ introduced in~\cite{BMmatroid}.  Let $\Bt = \bigoplus_{F,G \in \Lc} \B^G_F$ be the direct sum over all flats $F,G$ of $M$. Note that an independent set can be a basis for multiple distinct intervals of $M$.  Therefore, to avoid notational ambiguity, if $S$ is a basis for $M^G_F$ we write $e_{S,F}$ to be
the monomial $e_S$ as an element of $\B^G_F$.\footnote{The upper index $G = \overline{F \cup S}$ is determined by $F$ and $S$ and so there is no need to include it in the notation.}  Similarly, we may write $b_F$ to stress that $b$ is an element of $\Bt_F$.

We define $\B = \bigoplus_{F,G \in \F} \B^G_F$ to be the subspace of $\Bt$ obtained by considering only intervals between cyclic flats.  

The algebra $\Rc(M)$ is a subalgebra of $\End(\B)$ generated by various operators and their adjoints.  Similarly, we will show that $S(M)$ is isomorphic to a subalgebra of $\End(\Bt)$ generated by analogous operators and their adjoints.

To define the operators we make use of the following multiplication on $\Bt$ and its restriction to $\B$.  For any flats $F \subseteq G \subseteq H$, the union of a basis of $M^H_G$ and a basis of $M^G_F$ is a basis of $M^H_F$.  Thus we may define the multiplication
\[ \ast: \B^H_G \otimes \B^{G'}_F \to \B^H_F \]
induced by the wedge product on monomials if $G=G'$ and $0$ if $G \neq G'$.  

For any $b \in \Bt$, we define an operator $\tilde\mu_b$ on $\Bt$ given by
\[ \tilde\mu_b (x) = b_F \ast x, \]
 let $\tilde\kappa_b \in \End(\Bt)$ be the adjoint operator to $\tilde\mu_b$.  In particular, for any $b \in \B^G_F$, then $\tilde\mu_b(\B^K_L) \subseteq \B^G_L$ and is equal to zero if $K\neq F$.  Similarly, $\tilde\kappa_b(\B^K_L) \subseteq \B^F_L$ and is equal to zero if $K\neq G$

For cyclic flats $F,G \in \F$ and $b \in \B^G_F$ note that $\B \subseteq \Bt$ is a $\tilde\mu_b$- and $\tilde\kappa_b$-invariant subspace. We define $\mu_b$ and $\kappa_b, \in \End(\B)$ to be the restrictions of $\tilde\mu_b$ and $\tilde\kappa_b$ from $\Bt$ to $\cB$.

Let $\delta: \Lambda(M) \to \Lambda(M)$ be the differential defined by
\[ \delta(x) = \left(\sum_{s \in E} s \right) \wedge x. \]

Define $\cUc(M) := \Ker \left(\delta|_{\B(M)}: \Lambda^r_r(M) \to \Lambda^{r+1}_r(M) \right)$.  For any pair of flats $F \subseteq G$, we define a subspace $\cUc_F^G = \cUc(M^G_F)$ of $\B_F^G$.  Consider the subspaces 
\[\quad \cUc := \bigoplus_{F, G \in \F} \cUc_F^G \subseteq \B \quad\mathrm{and} \quad \tilde\cUc := \bigoplus_{F,G \in \Lc} \cUc_F^G \subseteq \Bt.\]

\begin{definition}\cite{BMmatroid} The matroidal Schur algebra $\Rc(M)$ is the subalgebra of $\End(\B)$ generated by the elements $\mu_\uc$ and $\kappa_\uc$ for all $\uc \in \cUc$. 
\end{definition}

We will show that $S(M)$ can be defined analogously:

\begin{theorem}\label{thm-isom} The algebra $S(M)$ is isomorphic to the subalgebra of $\End(\Bt)$ generated by the elements $\tilde\mu_\uc$ and $\tilde\kappa_\uc$ for all $\uc \in \tilde\cUc$.
\end{theorem}

To prove Theorem~\ref{thm-quiver}, describing $\Rc(M)$ in terms of $S(M)$, it will be helpful to use the following Reedy structure for $\Rc(M)$ shown in~\cite{BMmatroid}.
If $F,G \in \F$ are comparable cyclic flats, let $\Uc_{GF} \subset \End(\B)$ be the subspace
\[ \Uc_{GF} = \begin{cases} \{\mu_\uc~|~\uc \in \cUc^G_F\} & \mathrm{if~} F \subseteq G,~\mathrm{and} \\
 \{\kappa_\uc~|~\uc \in \cUc^F_G \} & \mathrm{if~} F \supseteq G.
\end{cases} \]

\begin{proposition}\cite[Corollary 1]{BMmatroid}\label{thm:cellular}
There is a direct sum decomposition
\[ \Rc(M) = \bigoplus_{F' \supseteq G\subseteq F} \Uc_{F'G} \Uc_{GF}, \]
where $E,F,G$ run over the set of cyclic flats of $M$.
\end{proposition}

Finally, we state our theorem relating the extended matroidal Schur algebra $S(M)$ to the original matroidal Schur algebras $\Rc(M)$ and $R(M)$.  

\begin{theorem}\label{thm-quiver}
Let $e = \sum_{K \in \F} (K)$ denote the sum of all such idempotents corresponding to cyclic flats.  The matroidal Schur algebra $\widecheck{R}(M) (\cong R(M^*))$ is isomorphic to 
\[\widecheck{R}(M) \cong eS(M) e/(\bigvee),\] 
where ($\bigvee$) is the relation:
\begin{itemize}
\item[($\bigvee$)] Let $K$ be a flat with a coloop $j \in K$ and let $L \lessdot K$ be the flat $K \setminus j$ obtained by removing $j$.  Then
\[ (K \gtrdot L \lessdot K) = (K).\]
\end{itemize}
\end{theorem}

\begin{example} Applying the theorem to $PG(n-1,q)$ and $M(K_n)$ we find:
\begin{enumerate} \item
The algebra $\Rc(PG(n-1,q))$ is isomorphic to \[eS(PG(n-1,q))e/(\bigvee),\] where $e$ is the idempotent defined as the sum of length one paths $(V)$ over all subspaces such that $\dim V \neq 1$ and $(\bigvee)$ is the relation:
\[ (V \gtrdot \{0\} \lessdot V) = (V), \]
for all 1-dimensional subspaces $V \subseteq \FM_q^n$.

\item
The algebra $\Rc(M(K_n))$ is isomorphic to $eS(M(K_n))e/(\bigvee)$, where $e$ is the idempotent defined as the sum of length one paths $(\sigma)$ over all partitions with no blocks of size two and $(\bigvee)$ stands for the ideal generated by the following relation:

Suppose that $\sigma \leq \pi$ and $\pi$ is obtained from $\sigma$ by merging two blocks $B=\{i\}, B'=\{j\}$ that each contain only one element.  Then:
\[ (\pi,\sigma,\pi) = (\pi). \]
Diagrammatically:
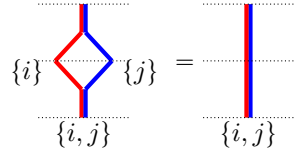
\[\begin{tikzpicture}[scale=0.75,baseline = -3]
	\draw[blue,-,ultra thick] (.5,0) to (0.035,-.5);
	\draw[blue,-,ultra thick] (.5,0) to (0.035,.5);
	\draw[red,-,ultra thick] (-.5,0) to (-.035,-.5);
	\draw[red,-,ultra thick] (-.5,0) to (-.035,.5);
	\draw[red,-,ultra thick] (-.035,.5) to (-.035,1);
	\draw[red,-,ultra thick] (-.035,-.5) to (-.035,-1);
	\draw[blue,-,ultra thick] (0.035,.5) to (0.035,1);
	\draw[blue,-,ultra thick] (0.035,-.5) to (0.035,-1);
	\draw[densely dotted] (-.8,-1) -- (.8,-1);
	\draw[densely dotted] (-.8,0) -- (.8,0);
	\draw[densely dotted] (-.8,1) -- (.8,1);
	\node at (-1,-0.2) {$\{i\}$};
        \node at (1,-0.2) {$\{j\}$};
	\node at (0,-1.3) {$\{i,j\}$};
\end{tikzpicture}
=
\begin{tikzpicture}[scale=0.75,baseline = -3]
	\draw[blue,-,ultra thick] (.035,1) to (0.035,-1);
	\draw[red,-,ultra thick] (-.035,1) to (-.035,-1);

	\draw[densely dotted] (-.8,-1) -- (.8,-1);
	\draw[densely dotted] (-.8,0) -- (.8,0);
	\draw[densely dotted] (-.8,1) -- (.8,1);
	\node at (0,-1.3) {$\{i,j\}$};
\end{tikzpicture}
\]
\end{enumerate}
\end{example}

%%%%%%%%%%%%%%%%%%%
\section{Faithful representation of $S(M)$}
\label{sec-faithful}
%%%%%%%%%%%%%%%%%%%

Our aim in this section is to prove Theorem~\ref{thm-isom}.  For each $F \in \Lc$, we defined an action of $S(M)$ on $\Bt_F$  in Section~\ref{sec-tilting}.  Consider the resulting $S(M)$-module structure on $\Bt = \bigoplus_{F \in \Lc} \Bt_F$ obtained by direct sum.  In the previous section $\phi$ denoted the action of $S(M)$ on $\Bt_F$ for a fixed flat $F$.  In this section, by a slight abuse of notation, we let $\phi: S(M) \to \End(\Bt)$ denote action on the direct sum $\Bt$. 

Let $\tilde\Rc(M) \subseteq \End(\Bt)$ denote the subalgebra generated by the elements $\tilde\mu_\uc$ and $\tilde\kappa_\uc$ for all $\uc \in \tilde\cUc$.

To prove Theorem~\ref{thm-isom} we must show that the image $\mathrm{Im}(\phi)$ is equal to $\tilde\Rc(M)$ and that $\phi$ is injective.

Note that for any flat $G$, $\phi(G) = \tilde\mu_{e_{\varnothing,G}}$ and $e_{\varnothing,G} \in \cUc_G^G$. Also for any $K,L \in \Lc$ such that $K \gtrdot L$,
\[ \phi(K \gtrdot L) = \tilde\mu_{d_L^K}, \quad \quad \phi(L \lessdot K) = \tilde\kappa_{d_L^K},\]
 and $d_L^K := \sum_{s \in K \setminus L} s \in \cUc_L^K$.\footnote{Here $\delta(d^K_L ) = d^K_L  \wedge d^K_L  = 0$, so $d^K_L$ is indeed in $\Ker(\delta) = \cUc^K_L$.}  Thus $\phi(G), \phi(K \gtrdot L), \phi(L \lessdot K) \in \tilde\Rc(M)$ and as the algebra $S(M)$ is generated by the idempotents $(G)$ for all $G \in \Lc$ and elements $(K\gtrdot L)$ and $(L \lessdot K)$ for all $K,L \in\Lc$ such that $K\lessdot L$ or $K\gtrdot L$, we conclude that $\mathrm{Im}(\phi) \subseteq \tilde\Rc(M)$.

\begin{lemma}\label{lem-surj}
The reverse inclusion, $\tilde\Rc(M) \subseteq\mathrm{Im}(\phi)$, also holds.
\end{lemma}

\begin{proof}
It suffices to show that for any $\uc \in \cUc_G^H$, $\tilde\mu_\uc, \tilde\kappa_\uc \in \mathrm{Im}(\phi)$.

For a flag $F = (K_k \gtrdot K_{k-1} \gtrdot \ldots \gtrdot K_0)$, observe that $\phi(F) = \tilde\mu_{d(F)}$ and $\phi(\iota F) = \tilde \kappa_{d(F)}$, where
\[ d(F) = d^{K_1}_{K_0} \wedge d^{K_2}_{K_1} \wedge \ldots \wedge d^{K_{k-1}}_{K_k}. \]

Let $\Fl^G_H$ denote the set of flags $F = (H = K_k \gtrdot K_{k-1} \gtrdot \ldots \gtrdot K_0 = G)$.  It is enough to show that $\cUc^G_H$ is equal to the span of 
\[ \{d(F)~|~F \in \Fl^H_G\} .\]
For each $F = (H=K_k \gtrdot K_{k-1} \gtrdot \ldots \gtrdot K_0=G)\in \Fl^H_G$, we have $d_G^H:= \sum_{s \in H \setminus G} s = \sum_{i=1}^{k-1} d^{K_i}_{K_{i+1}} $ and so:
\[ \delta(d(F)) = d_G^H \wedge d(F) = \left( \sum_{i=1}^{k} d^{K_{i}}_{K_{i-1}} \right) \wedge d^{K_0}_{K_1} \wedge \ldots \wedge d^{K_{k-1}}_{K_k}=0, \] so $d(F) \in \cUc_G^H$ for all $F \in \Fl_G^H$.  

It remains the show that $\{d(F)~|~F \in \Fl_G^H\}$ spans $\cUc_G^H$.  
Let $\ul{\B}_G^H \subseteq \B_G^H$ denote the free $\Z$-submodule generated by the set of monomials $e_B$ for the nbc bases $B$ of $M_G^H$.    
Recall that the rank of the free $\Z$-module $\cUc_G^H$ is equal to the number of nbc bases of $M_G^H$ (see \cite[Lemma 6]{BMmatroid}).  To show that $\{d(F)~|~F \in \Fl^G_H\}$ spans $\cUc_G^H$ it is therefore enough to construct a subset of $\{d(F)~|~F \in \Fl^G_H\} \subset \cUc_G^H$ that is sent to a basis for $\ul{\B}_G^H$ by the orthogonal projection $\B_G^H = \ul\B_G^H \oplus (\ul\B_G^H)^\perp \to \ul\B_G^H$.

For each nbc basis $B = \{x_1>x_2> \ldots >x_r\}$ for $M_G^H$, let $F = (K_k \gtrdot K_{k-1} \gtrdot \ldots \gtrdot K_1 \gtrdot K_0=G)$ be the flag with $K_i = \overline{\{x_1,\ldots,x_i\}}$.  Consider the wedge product $d_F = d^{K_0}_{K_1} \wedge d^{K_1}_{K_2} \wedge  \ldots \wedge d^{K_{r-1}}_{K_r} \in \cUc(M)$.  Suppose that there exists $y_i \in K_i \setminus K_{i-1}$ such that $y_i <x_i$.  Then $\{x_1, x_2, \ldots, x_i, y_i\}$ is dependent and contains a circuit $C$ with minimal element $y_i$, which implies that $\{x_1, x_2, \ldots, x_i\}$ contains a broken circuit $C \setminus y_i$.  This contradicts the assumption that $B$ is nbc, so we conclude that for each $i$, $x_i$ is the minimal element of $K_i \setminus K_{i-1}$.  

Thus $d(F) \in \cUc$ is of the form $d(F) = \pm e_B + \sum_{B'>B} a_{B'} e_{B'}$ for some coefficients $a_{B'} \in \Z$ and by induction we conclude that $e_B \in \ul{\B}^\perp + \cUc$.  Letting $B$ run over all nbc bases we have then constructed a sequence of wedge products $d_F$ that project to a basis for $\ul\B$.  We conclude that this set of wedge products is a basis of $\cUc(M)$ and so $\phi$ is surjective.
\end{proof}

\begin{corollary}\label{cor-isom}
The map from $\Z \cdot \Fl_G^H$ to $\cUc_G^H$ defined on flags by
\[F = (H = K_k  \gtrdot \ldots \gtrdot K_0 = G) \to d(F) = d^{K_0}_{K_1} \wedge d^{K_1}_{K_2} \wedge \ldots \wedge d^{K_{k-1}}_{K_k} \in \cUc^G_H,\] 
induces an isomorphism $(H) \Fs^+ (G) \to \tilde\Uc_{HG}:= \{\tilde\mu_\uc~|~\uc \in \cUc^H_G\}$, which is the restriction of $\phi$ to $(H) \Fs^+ (G) \subset S(M)$.
\end{corollary}

\begin{proof}
The map $\uc \mapsto \tilde\mu_\uc$ is an isomorphism from $\cUc_G^H$ to $\Uc_{HG}$, with inverse $x \mapsto x(e_{\varnothing,G})$.  The result of Section~\ref{sec-tilting} implies that the map $F \mapsto d(F)$ descends to the restriction of $\phi$.  The proof of the previous result shows that the restriction of $\phi$, $(H) \Fs^+ (G) \to \tilde\Uc_{HG}$, is surjective.  A surjective homomorphism of free abelian groups of the same rank is an isomorphism, so to complete the argument we need to show that the flag space $(H) \Fs^+ (G)$ is a free $\Z$-module of the same rank as $\cUc_G^H$. Indeed, this is true and both have rank $\mu^+(M^H_G)$ by \cite[Lemma 6]{BMmatroid}) and~\cite[Proposition (2.18)]{BV}.
\end{proof}

\begin{proposition}
$\phi$ is injective.
\end{proposition}

\begin{proof}
Recall that any element of $S(M)$ can be expressed as linear combination of concatenations of a positive and negative path:
\[(J \gtrdot \ldots \gtrdot L  \lessdot \ldots \lessdot K).\]  

Suppose that $\phi\left(\sum_i a_i F_i \right)= 0$, where each $F_i = F_i^+ \cdot F_i^-$ is a concatenation of a positive and negative path.   For any $J,K \in \Lc$,
\[ (K) \cdot \phi\left(\sum_i a_i F_i \right) \cdot (J) = \phi\left(\sum_i a_i (K)\cdot F_i \cdot(J)\right)= 0.\]
In this way we can assume that each $F_i$ begins with the same flat $J$ on the right and ends at the same last flat $K$ on the left.

For each flag path $F_i$ in the sum, let $L_i$ be the minimal flat of $F_i$.  If the set of $F_i$ is empty, then the sum is zero and we are done, so suppose the set of $F_i$ is nonempty and let $L$ be a maximal element among the set of all $L_i$.

Consider the restriction of $\phi(\sum_i a_i F_i) = \sum_i a_i \phi(F_i) \in \End(\Bt)$ to $\B_L^J$.  For every $i$, $\phi(F_i)(\B_L^J) \subseteq \B_L^K$ and $\phi(F_i)|_{\B_L^J}$ factors through $\B_L^{L_i}$.  By definition, $\B_L^{L_i} =0$ unless $L_i \geq L$, but by our assumption $L \leq L_i$, so $\phi(F_i)(\B^L_J)=0$ unless $L=L_i$.

Let $I_L$ denote the set of $i$ such that $L_i=L$.  We can then identify the sum $\sum_{i\in I_L} a_i F_i$ with an element of $(K)\Fs^+(L) \otimes (L)\Fs^- (J)$.  

Consider the map from $(K)\Fs^+(L) \otimes (L)\Fs^- (J)$ to $\Hom(\cB^L_J, \cB^L_K)$ defined by $F^+ \otimes F^- \mapsto \phi(F^+ \cdot F^-)|_{\B^L_J}$.  We claim that this map is injective.  To see this, note that it factors as:
\[ (K) \Fs^+ (L) \otimes (L) \Fs^- (J) \stackrel{\mathrm{id}\otimes \iota}{\cong} \cUc_L^K \otimes \cUc^J_L \hookrightarrow  \cB_L^K \otimes \cB_L^J \cong \Hom(\cB_L^J, \cB_L^K), \]
where the first isomorphism is induced by Corollary~\ref{cor-isom} and the map in the middle is induced by the inclusions $\cUc \hookrightarrow \cB$ and the final isomorphism is from the identification $\cB^L_J \cong (\cB^L_J)^*$ coming from the pairing $\langle~,~\rangle$.  Thus the composition is injective and $\sum_{i \in I_L} a_i F_i=0$.  Removing these terms and repeating the same argument, we conclude that $\sum_i a_i F_i=0$.
\end{proof}

%%%%%%%%%%%%%%%%%%%
\section{Presentation of original matroidal Schur algebras}
\label{sec-Rcheck}
%%%%%%%%%%%%%%%%%%%

We conclude by proving Theorem~\ref{thm-quiver}, which gives a finite presentation of the original Schur algebra $\widetilde{R}(M)$, and hence by duality of $R(M) = \widetilde{R}(M^*)$.

Consider the restriction of the map $\phi: S(M) \to \End(\tilde\B)$ to $e S(M) e$.  For any path $F$ that begins and ends at cyclic flats, the subspace $\B \subseteq \tilde\B$ is invariant for the operator $\phi(F) \in \tilde\Rc \subset \End(\tilde\B)$ and so $\phi(F)$ restricts to an operator $\phi'(F) = \phi(F)|_\B$.  Thus $\phi$ induces a map $\phi':eS e \to \End(\B)$.

\begin{lemma}
The image $\im(\phi')$ is equal to $\Rc(M)$.
\end{lemma}

\begin{proof}
We first show that $\Rc(M) \subseteq \im(\phi')$.  The subalgebra $\Rc(M) \subseteq \End(\B)$ is generated by the operators $\mu_\uc = \tilde\mu_\uc |_{\B}$ and $\kappa_\uc = \tilde\kappa_\uc |_{\B}$ for all $\uc \in \cUc$. By the proof of Lemma~\ref{lem-surj},  $\tilde\mu_\uc |_{\B}$ (resp. $\tilde\kappa_\uc |_{\B}$) for any element $\uc \in \cUc^G_F$ for $F,G \in \F$ can be expressed as the image under $\phi'$ of a linear combination of positive (resp. negative) flag paths of the form $(G \gtrdot \ldots \gtrdot F)$ (resp. $(F \lessdot \ldots \lessdot G)$).  Thus $\Rc(M) \subseteq \im(\phi')$.

We now show that $\im(\phi') \subseteq \Rc(M)$.  Any element of $eS(M) e$ is a linear combination of flag paths $F = F^+ \cdot F^- = (J \gtrdot \ldots \gtrdot L \lessdot \ldots \lessdot K)$ such that $J,K \in \F$, so to prove the lemma it suffices to show that the image under $\phi'$ of such a path is contained in $\Rc(M)$.
If $L \in \F$, then 
\[ \phi'(F^+ F^-)  = \tilde\mu_{d(F^+)} \tilde\kappa_{d(F^-)}|_\B = \mu_{d(F^+)} \kappa_{d(F^-)} \in \Rc(M).\]
If $L \not\in \F$, choosing any order on the coloops of $L$ and removing them one by one gives a positive flag $H = (L=L_k \gtrdot  \ldots \gtrdot L_1 \gtrdot L_0=L')$ with $L' \in \F$.

Then
\[ \phi'(F^+ \cdot F^-) = \tilde\mu_{d(F^+)} \tilde\kappa_{d(F^-)}|_\B = \tilde\mu_{d(F^+)} (\tilde\mu_{d(H)} \tilde\kappa_{d(\iota H)}) \tilde\kappa_{d(F^-)}|_\B \]
\[= \tilde\mu_{d(F^+ \cdot H)}  \tilde\kappa_{d(\iota H \cdot F^-)}|_\B = \mu_{d(F^+ \cdot H)} \kappa_{d(\iota H \cdot F^-)}\in \Rc,\]
here the second equality holds because for any $G \subset L_i$, $\phi(L_{i+1} , L_i)=\tilde\mu_{d_{L_{i+1}}^{L_i}}:\B_G^{L_{i}} \to \B_G^{L_{i+1}}$ and $\phi(L_i , L_{i+1}) = \tilde\kappa_{d^{L_i}_{L_{i+1}}}:\B_G^{L_{i+1}} \to \B_G^{L_{i}}$ are inverses. 
\end{proof}

It remains to show that the kernel of the homomorphism $\phi': eS e  \onto \Rc(M)$ is generated by the relation ($\bigvee$). 

Now consider the free $\Z$-submodule \[\bigoplus_{L \subseteq K,J \in \F} (K) \Fs^+ (L) \otimes (L) \Fs^- (J) \subseteq eS(M)e\] where the sum runs over cyclic flats $J,K,L$.  The composition of this inclusion with the projection to $eS e/(\bigvee)$ is surjective (because any path that begins and ends at a cyclic flat, modulo the relation ($\bigvee$) can be expressed as a path of the form $F^+ F^-$, where $F^+ \in e \F^+ e, F^- \in e \F^- e$).

Thus we have a commutative diagram:
\[\begin{tikzcd}
& eS e \arrow[d]\arrow[dr,twoheadrightarrow,"\phi'"] & \\
\bigoplus_{L \subset K,J \in \F} (K) \Fs^+ (L) \otimes (L) \Fs^- (J) \arrow[ur,hook] \arrow[r,twoheadrightarrow]&  eS e/(\bigvee) \arrow[r,twoheadrightarrow] & \Rc(M).
\end{tikzcd}
\]

Now the rank \[ \mathrm{rank}\left(\bigoplus_{L \subseteq K,J \in \F} (K) \Fs^+ (L) \otimes (L) \Fs^- (J) \right)= \sum_{L \subseteq K,J \in \F} \mu^+(M^K_L) \cdot \mu^+(M^J_L).\] Similarly, $\Rc(M) \cong \bigoplus_{L \subseteq K,J \in \F} \Uc_{KL} \otimes \Uc_{LJ}$ by Proposition~\ref{thm:cellular}, which has the same rank by~\cite[Proposition (2.18)]{BV}.  We conclude that both maps along the bottom of the diagram are isomorphisms as $\Z$-modules.

\bibliography{refs}
\bibliographystyle{amsalpha}

\end{document}